\documentclass[a4paper,12pt]{article}
\usepackage[utf8]{inputenc}
\usepackage{amsthm}
\usepackage{amsfonts}
\usepackage{amssymb}
\usepackage{hyperref}
\usepackage{graphicx}
\usepackage{rotating}
\usepackage{wrapfig}
\usepackage[all]{xy}
\usepackage{pgf,tikz}
\usetikzlibrary{arrows}
\usetikzlibrary[patterns]

\newcommand{\N}{\mathbb{N}}

\newcommand{\R}{\mathbb{R}}

\newcommand{\K}{\mathbb{K}}

\newcommand{\U}{\mathfrak{U}}
\newcommand{\B}{\mathfrak{B}}

\newtheorem{prop}{Proposition}[section]
\newtheorem{cor}[prop]{Corollary}

\newtheorem{teo}[prop]{Theorem}

\newtheorem{lema}[prop]{Lemma}

\theoremstyle{definition}

\newtheorem*{ex}{Example}
\newtheorem{defi}[prop]{Definition}
\newtheorem*{obs}{Remark}
\newtheorem*{acknowledgements}{Acknowledgements}

\title{Glueing spaces without identifying points}

\author{Lucas H. R. de Souza}

\begin{document}

\maketitle

\def\eod{\hfill$\square$}

\begin{abstract}\, \,  In this paper we develop the theory of Artin-Wraith glueings for topological spaces. As an application, we show that some categories of compactifications of coarse spaces that agree with the coarse structures are invariant under coarse equivalences. As a consequence, if X and Y are some well behaved metric spaces that are coarse equivalent, then they have the same space of ends (generalizing the well known fact that works on quasi-isometric proper geodesic metric spaces). As another application, we show that for every compact metrizable space $Y$, there exists only one, up to homeomorphisms, compactification of the Cantor set minus one point such that the remainder is homeomorphic to $Y$.
\end{abstract}

\let\thefootnote\relax
\let\thefootnote\relax

\tableofcontents

\section*{Introduction}

Let $X$ and $Y$ be topological spaces and maps $f: Closed(X) \rightarrow Closed(Y)$ and $g: Closed(Y) \rightarrow Closed(X)$ that preserves the empty set and finite unions, $\forall A \in Closed(X)$, $g \circ f (A) \subseteq A$ and $\forall B \in Closed(Y)$, $f \circ g (B) \subseteq B$. From $f$ and $g$, we construct a topological space $X+_{f,g}Y$ where the set is $X\dot{\cup} Y$ and it extends both topologies. Reciprocally, if a space $Z$ is the union of two disjoint subspaces $X$ and $Y$, then there exists a unique pair of maps $f$ and $g$ with those properties such that $Z = X+_{f,g}Y$.

This construction is called Artin-Wraith glueing. This appears in Topos Theory in a completely analogous way. It appears in \cite{KP} in its full generality. In the topological case that we are interested in, the special case where $g$ is a constant map equal to the empty set is equivalent to say that the space $X$ is open in $X+_{f,g}Y$. It turns out to be a convenient tool to work with compactifications of locally compact spaces. This construction is used in \cite{Ge2} in the proof of the existence of the Attractor-Sum compactification of a group that acts on a Hausdorff compact space with the convergence property (i.e. a compactification of the group with the boundary equal to the space that the group acts on and such that the induced action still has the convergence property).

Our first objective with this paper is to develop the theory Artin-Wraith glueing for the context of topological spaces. Due to a lack of knowledge on Topos Theory from the author, it is not clear which propositions on the sections $3$, $4$ and $5$ are already proved in the context of Topos Theory. For the sake of completeness, we are proving all propositions anyway. We used that construction on two preprints \cite{So} and \cite{So3} of Geometric Group Theory to do a correspondence theory of perspective compactifications on the first one and to blow up bounded parabolic points on perspective compactifications on the second one. These tools developed in the present paper are quite useful to simplify constructions of some topological spaces and to give quite simple proofs of continuity of some maps. Actually, let's consider a map $X+_{f,g}Y \rightarrow Z+_{h,j}W$ that sends $X$ to $Z$ and $Y$ to $W$. If it is continuous when restricted to $X$ and when restricted to $Y$, then the continuity of the whole map is equivalent to a diagram problem (\textbf{Proposition \ref{mapacontinuo}}).

Our second objective with this paper is to give some applications of this theory of Artin-Wraith glueings to compactifications of Hausdorff locally compact spaces. From a proper map $X \rightarrow Y$, where $X$ and $Y$ are Hausdorff locally compact spaces, we are able to transfer functorialy a compactification of $Y$ to a compactification of $X$, preserving remainders (\textbf{Proposition \ref{pullbackdecompactificacoes}}). This has four major consequences:

\begin{enumerate}
    \item If $(X,\varepsilon)$ and $(Y,\zeta)$ are locally compact paracompact Hausdorff spaces with suitable coarse structures (in the sense of John Roe's book \cite{Ro}) and if they are coarse equivalent, then there exists a correspondence between the metrizable compactifications of $X$ that agrees with $\varepsilon$ and metrizable compactifications of $Y$ that agrees with $\zeta$ (Theorem \ref{Teorema2}). Furthermore, if such coarse equivalence is continuous with continuous quasi-inverse, then there exists a correspondence between the compactifications of $X$ that agrees with $\varepsilon$ and compactifications of $Y$ that agrees with $\zeta$ (Theorem \ref{Teorema1}). The same construction was used by Guilbault and Moran \cite{GM} on their Boundary Swapping Theorem to $E\mathcal{Z}$-structures of groups and ours is a direct generalization. In \cite{So} we also have a similar construction for groups that generalizes a construction from Gerasimov's Attractor-Sum Theorem \cite{Ge2}. In the future we intend to unify those two constructions.
    \item If $(X, d)$ and $(Y,d')$ are proper metric spaces that are locally connected and perspectively connected (see \textbf{Definition \ref{perspectivelyconnected}}) and are coarse equivalent, then their spaces of ends are homeomorphic (\textbf{Proposition \ref{coarseequivimpliesendshomeo}}). This result is well known for proper geodesic spaces that are quasi-isometric (Proposition 8.29 of \cite{BH}).
    \item If $X$ is the coproduct of the Hausdorff compact spaces $\{C_{i}\}_{i\in \Gamma}$, then the category of compactifications of $\Gamma$ (with the discrete topology) is isomorphic with a full subcategory of the compactifications of $X$ (\textbf{Theorem \ref{teoremaprincipal}}).
    \item If $Z$ is a compact metrizable space, then there exists, up to homeomorphisms, a unique compactification of the Cantor set minus one point such that the remainder is homeomorphic to $Z$ (\textbf{Theorem \ref{Cantor}}). This is an analogue to a theorem of Pelczy\'nski that does the same thing for compactifications of the natural numbers with the discrete topology (p. 87 of \cite{Pe}).
\end{enumerate}

\begin{acknowledgements}This paper contains part of my Masters' thesis and my PhD thesis. Both 
thesis were written under the advisorship of Victor Gerasimov, to whom I am
grateful for our several discussions and lots of things that I’ve learned during
this meantime.

I would like to thank Peter Faul who pointed out to me the existence of Artin-Wraith glueings.

\end{acknowledgements}

\section{Preliminaries}

This section contains some notation and well known results that are used through this paper.

We use the symbol $\Box$, besides its usual purpose, at the end of a proposition, lemma or theorem to say that its proof follows immediately from the previous considerations.

\begin{prop}(RAPL - Right Adjoints Preserves Limits, Proposition 3.2.2 of \cite{Bo}) Let $F: \mathcal{C} \rightarrow \mathcal{D}$ and $G: \mathcal{D} \rightarrow \mathcal{C}$ be two functors with $G$ adjoint to $F$. If $H: \mathcal{E} \rightarrow \mathcal{D}$ is a functor that possesses a limit, then $\lim\limits_{\longleftarrow} (G\circ H)$ exists and it is equal to $G(\lim\limits_{\longleftarrow} H)$.
\end{prop}

\subsection{Uniform structures}

\begin{defi}Let $X$ be a set, $Y \subseteq X$ and $u \subseteq X \times X$. We define the $u$-neighborhood of $Y$ by $\mathfrak{B}(Y,u) = \{x \in X: \exists y \in Y: \ (x,y)\in u\}$. We say that $Y$ is $u$-small if $Y\times Y \subseteq u$ and we denote by $Small(u)$ the set of all subsets of $X$ that are $u$-small.
\end{defi}

\begin{prop}\label{small}Let $f:(X_{1},\U_{1}) \rightarrow (X_{2},\U_{2})$ be a uniformly continuous map, $u \in \U_{2}$ and $Y \subseteq X_{2}$. If $Y \in Small(u)$, then $f^{-1}(Y) \in Small((f^{2})^{-1}(u))$. If $f$ is surjective, then the converse is also true.\eod
\end{prop}

\begin{prop}(Proposition 10, $\S2.7$, Chapter 2 of \cite{Bou}) Let $\Gamma$ be a directed set, $\{(X_{\alpha},\U_{\alpha}), f_{\alpha_{1}\alpha_{2}}\}_{\alpha,\alpha_{1},\alpha_{2} \in \Gamma}$ an inverse system of uniform spaces, $B_{\alpha}$ a basis for $\U_{\alpha}, \ (X,\U) = \lim\limits_{\longleftarrow} X_{\alpha}$ and $\pi_{\alpha}: X \rightarrow X_{\alpha}$ the projection maps. Then, the set $\{(\pi_{\alpha}^{2})^{-1}(b): \alpha \in \Gamma, \ b \in B_{\alpha}\}$ is a basis for $\U$.
\end{prop}

\subsection{Topology}

\begin{defi}Let $X$ be a topological space. A family $\{F_{\alpha}\}_{\alpha \in \Gamma}$ of subsets of $X$ is locally finite if $\forall x \in X, \ \exists U$ a neighborhood of $x$ such that $U\cap F_{\alpha} \neq \emptyset$ only for a finite subset of $\Gamma$.
\end{defi}

\begin{prop}(Proposition 4, $\S2.5$, Chapter 1 of \cite{Bou}) Let $X$ be a topological space and $\{F_{\alpha}\}_{\alpha \in \Gamma}$ a locally finite family of closed sets of $X$. Then, $\bigcup_{\alpha \in \Gamma} F_{\alpha}$ is closed.
\end{prop}

\begin{prop}\label{liftnet}Let $f: X \rightarrow Y$ be a continuous closed map. If $\{x_{i}\}_{i \in \N}\subseteq X$ is a net such that $\{f(x_{i})\}_{i \in \N}$ converges to $y \in Y$ and $f^{-1}(y)$ is a single point, then $\{x_{i}\}_{i \in \N}$ converges to $f^{-1}(y)$. \eod
\end{prop}

\begin{prop}(Proposition 9, $\S4.4$, Chapter 1 of \cite{Bou}) Let $\Gamma$ be a directed set, $\{X_{\alpha}, f_{\alpha_{1}\alpha_{2}}\}_{\alpha,\alpha_{1},\alpha_{2} \in \Gamma}$ an inverse system of topological spaces, $B_{\alpha}$ a basis for $X_{\alpha}, \ X = \lim\limits_{\longleftarrow}X_{\alpha}$ and $\pi_{\alpha}: X \rightarrow X_{\alpha}$ the projection maps. Then, the set $\{\pi_{\alpha}^{-1}(b): \alpha \in \Gamma, \ b \in B_{\alpha}\}$ is a basis for $X$.
\end{prop}

\begin{prop}(Aleksandrov Theorem, $\S 10.4$, Chapter I of \cite{Bou}) Let $X$ be a Hausdorff compact space and $\sim$ an equivalence relation on $X$. Then $X/\sim$ is Hausdorff if and only if $\sim$ is closed in $X^{2}$.
\end{prop}

\begin{prop}(Corollary 3.1.20 of \cite{En}) Let $X$ be a Hausdorff compact space, $m$ an infinite cardinal and $\{X_{\alpha}\}_{\alpha \in \Gamma}$ a family of subspaces of $X$ such that $X = \bigcup_{\alpha \in \Gamma} X_{\alpha}, \ \# \Gamma \leqslant m$ and $\forall \alpha \in \Gamma, \ \omega(X_{\alpha}) \leqslant m$ (where $\omega(Y)$ is the lowest cardinality of a basis of $Y$). Then $\omega(X) \leqslant m$.
\end{prop}

\begin{cor}\label{uniaometrizavel}Let $X$ be a Hausdorff compact space where there exists a family of subspaces $\{X_{n}\}_{n\in \N}$ such that each one has a countable basis and $X = \bigcup_{n\in \N}X_{n}$. Then $X$ is metrizable. \eod
\end{cor}

\begin{cor}Let $X$ be a Hausdorff locally compact space with countable basis and $Z$ be a metrizable compact space. If $Y$ is a compactification of $X$ such that the remainder is homeomorphic to $Z$. Then $Y$ is metrizable. \eod
\end{cor}

\begin{obs}Compactifications for us are always Hausdorff spaces such that the original space is a dense subspace, unless we say otherwise.
\end{obs}

\subsection{Coarse geometry}

The next definitions and propositions follow John Roe's book \cite{Ro}.

\begin{defi}Let $X$ be a set. A coarse structure on $X$ is a set $\varepsilon \subseteq X \times X$ satisfying:

\begin{enumerate}
    \item The diagonal $\Delta X \in \varepsilon$,
    \item If $e \in \varepsilon$ and $e' \subseteq e$, then $e' \in \varepsilon$,
    \item If $e,e' \in \varepsilon$ then $e\cup e' \in \varepsilon$,
    \item If $e \in \varepsilon$  then $e^{-1} = \{(a,b): (b,a) \in e\} \in \varepsilon$,
    \item If $e,e' \in \varepsilon$  then $e' \circ e = \{(a,b): \exists c \in X: (a,c) \in e, (c,b) \in e'\} \in \varepsilon$.
\end{enumerate}

\end{defi}

\begin{defi}We say that a subset $B$ of a coarse space $(X,\varepsilon)$ is bounded if $B\times B \in \varepsilon$.
\end{defi}

\begin{defi}Let $X$ be a topological space. A subset of $X$ is topologically bounded (or relatively compact) if its closure on $X$ is compact. We say that $(X,\varepsilon)$ is a proper coarse space if the coarse structure has a neighborhood of the diagonal $\Delta X$ and every bounded subset of $X$ is topologically bounded.
\end{defi}

\begin{defi}A coarse space $(X,\varepsilon)$ is coarsely connected if for every $(x,y) \in X \times X$, $\exists e \in \varepsilon:$ $(x,y) \in e$.
\end{defi}

\begin{defi}Let $X$ be a topological space. A subset $e \subseteq X\times X$ is proper if $\forall B \subseteq X$ topologically bounded, $\B(B,e)$ and $\B(B, e^{-1})$ are topologically bounded.
\end{defi}

\begin{defi}Let $(M,d)$ be a metric space. The bounded coarse structure associated to the metric $d$ is the collection of sets $e \subseteq X\times X$ such that $\sup\{d(x,y): (x,y)\in e\} < \infty$. We denote this coarse structure by $\varepsilon_{d}$.
\end{defi}

\begin{obs}If $(M,d)$ is a proper metric space (in the sense that every closed ball is compact), then $(M,\varepsilon_{d})$ is a coarsely connected proper coarse space. \end{obs}

\begin{prop}\label{propercontrolled}(Proposition 2.23 of \cite{Ro}) Let $(X,\varepsilon)$ be a coarsely connected proper coarse space. A subset of $X$ is bounded if and only if it is topologically bounded. Moreover, every element of $\varepsilon$ is proper.
\end{prop}

\begin{defi}A map $f: (X,\varepsilon) \rightarrow (Y,\zeta)$ is coarse if $\forall e \in \varepsilon$, $f(e)\in \zeta$ and $\forall B \subseteq Y$ bounded, $f^{-1}(B)$ is bounded.
\end{defi}

\begin{defi}Let $S$ be a set and $(X,\varepsilon)$ a coarse space. Two maps $f,g: S \rightarrow X$ are close if $\{(f(s),g(s)): s \in S\} \in \varepsilon$.
\end{defi}

\begin{defi}Two coarse spaces $(X,\varepsilon)$ and $(Y,\zeta)$ are coarsely equivalent if there exists two coarse maps $f: X \rightarrow Y$ and $g: Y \rightarrow X$ that are quasi-inverses, i.e. $f\circ g$ is close to $id_{Y}$ and $g \circ f$ is close to $id_{X}$. A coarse map $f: (X,\varepsilon) \rightarrow (Y,\zeta)$ is a coarse embedding if it is a coarse equivalence between $X$ and $f(X)$, with the coarse structure given by $\{A \in \zeta: A \subseteq f(X)^{2}\}$.
\end{defi}

\begin{prop}(Theorem 2.27 of \cite{Ro}) Let $X$ be a locally compact paracompact space, $W$ a compactification of $X$, $\partial X = W - X$ and $e \subseteq X \times X$. The following conditions are equivalent:

\begin{enumerate}
    \item $Cl_{W^{2}}(e) \cap (W^{2} - X^{2}) \subseteq \Delta \partial X$.
    \item $e$ is proper and if  $\{(x_{\gamma},y_{\gamma})\}_{\gamma \in \Gamma}$ is net contained in $e$ such that $\lim x_{\gamma} = x \in \partial X$, then $\lim y_{\gamma} = x$.
    \item $e$ is proper and $\forall x \in \partial X$, $\forall V$ neighborhood of $x$ in $W$, there exists $U$ a neighborhood of $x$ such that $U \subseteq V$ and $e \cap (U \times (X-V)) = \emptyset$.
\end{enumerate}

We say that $e$ is perspective if it satisfies these equivalent definitions. We denote by $\varepsilon_{W}$ the set of perspective subsets of $X \times X$. Then $(X,\varepsilon_{W})$ is a coarsely connected proper coarse space.
\end{prop}

\section{Artin-Wraith glueing}

\subsection{Construction}

This section is entirely analogous to Theorem 1.1 of \cite{KP}.

\begin{defi}Let $X$ and $Y$ be topological spaces. We say that a map $f: Closed(X) \rightarrow Closed (Y)$ is admissible if  $\forall A,B \in Closed(X)$, $f(A\cup B) = f(A)\cup f(B)$ and $f(\emptyset) = \emptyset$. Let's fix an admissible map $f$. Let's declare $A \subseteq X\dot{\cup} Y$ as a closed set if $A \cap X \in Closed(X), \ A \cap Y \in Closed(Y)$ and $f(A\cap X) \subseteq A$. Therefore, let's denote by $\tau_{f}$ the set of the complements of this closed sets and $X+_{f}Y = (X\dot{\cup} Y,\tau_{f})$.
\end{defi}

\begin{prop}Actually, $\tau_{f}$ is a topology.
\end{prop}

\begin{proof}We have that $(X\cup Y)\cap X = X \in Closed(X), \ (X\cup Y)\cap Y = Y \in Closed(Y)$ and $f((X\cup Y)\cap X) = f(X) \subseteq X\cup Y$. So $X\cup Y \in Closed(X+_{f}Y)$.

We have also that $\emptyset \cap X = \emptyset \in Closed(X), \ \emptyset\cap Y = \emptyset \in Closed(Y)$ and $f(\emptyset \cap X) = f(\emptyset) = \emptyset$. So $\emptyset \in Closed(X+_{f}Y)$.

If $A,B \in Closed(X+_{f}Y)$, then $(A\cup B)\cap X = (A\cap X)\cup (B \cap X) \in Closed(X)$, $(A\cup B)\cap Y = (A\cap Y)\cup (B \cap Y) \in Closed(Y)$ and $f((A\cup B)\cap X) = f((A\cap X)\cup (B \cap X)) = f(A\cap X)\cup f(B \cap X) \subseteq A \cup B$ (because $f(A\cap X)\subseteq A$ and $f(B \cap X) \subseteq B$). So $A\cup B \in Closed(X+_{f}Y)$.

Finally, let $\{A_{i}\}_{i\in \Gamma}$ be a family of closed sets. Then $(\bigcap\limits_{i\in \Gamma}A_{i}) \cap X = \bigcap\limits_{i\in \Gamma}(A_{i}\cap X) \in Closed(X)$, because each $A_{i}\cap X \in Closed(X)$. Analogously, $(\bigcap\limits_{i\in \Gamma}A_{i}) \cap Y \in Closed(Y)$. And $\forall i \in \Gamma, \ f((\bigcap\limits_{i\in \Gamma}A_{i}) \cap X) \subseteq f(A_{i}\cap X) \subseteq A_{i}$, which implies that $f((\bigcap\limits_{i\in \Gamma}A_{i}) \cap X)\subseteq \bigcap\limits_{i\in \Gamma}A_{i}$. So $\bigcap\limits_{i\in \Gamma}A_{i} \in Closed(X+_{f}Y)$. \end{proof}

For our purposes it is only necessary this definition of sum of spaces. However, for the sake of completeness, we present the more general definition and also develop the theory about it.

\begin{defi}Let $X$ and $Y$ be topological spaces. We say that two maps $f: Closed(X) \rightarrow Closed (Y)$ and $g: Closed(Y) \rightarrow Closed (X)$ are an  admissible pair if $f$ and $g$ are admissible maps, $\forall A \in Closed(X)$, $g \circ f(A) \subseteq A$ and $\forall B \in Closed(Y)$, $f \circ g(B) \subseteq B$. We denote the topological space $(X\dot{\cup} Y,\tau_{f}\cap \tau_{g})$ by $X+_{f,g}Y$.
\end{defi}

This is an Artin-Wraith glueing of these two spaces. We call it sum of spaces, for short. When $\emptyset: Closed(Y) \rightarrow Closed(X)$ is the constant map to the empty set, we have that $X+_{f,\emptyset}Y = X+_{f}Y$. Note that for every admissible map $f$, the pair $f,\emptyset$ is always admissible.

We have that $D \in Closed(X+_{f,g}Y)$ if and only if $D \cap X \in Closed(X)$, $D \cap Y \in Closed(Y)$, $f(D\cap X) \subseteq D$ and $g(D\cap Y) \subseteq D$.

\begin{prop}Let $A \in Closed (X)$ and $B \in Closed(Y)$. Then $Cl_{X+_{f,g}Y}A = A\cup f(A)$ and $Cl_{X+_{f,g}Y}B = B\cup g(B)$.
\end{prop}

\begin{proof}We have that $(A \cup f(A)) \cap X = A \in Closed (X)$, $(A \cup f(A)) \cap Y = f(A) \in Closed (Y)$, $f((A\cup f(A))\cap X) = f(A) \subseteq A \cup f(A)$ and $g((A\cup f(A))\cap Y) = g(f(A)) \subseteq A \subseteq A \cup f(A)$. So $A\cup f(A) \in Closed(X+_{f,g}Y)$.

Let $D\in Closed(X+_{f}Y)$ such that $A \subseteq D$. We have that $f(D\cap X)\subseteq D$. But $f(D\cap X) = f((A\cup D)\cap X) = f(A\cap X) \cup f(D\cap X) = f(A) \cup f(D\cap X)$, which implies that $f(A) \subseteq D$. So $A \cup f(A) \subseteq D$.

Thus, $Cl_{X+_{f,g}Y}A = A\cup f(A)$. Analogously $Cl_{X+_{f,g}Y}B = B\cup g(B)$
\end{proof}

\begin{cor}$X$ is dense in $X+_{f,g}Y$ if and only if $f(X) = Y$.
\end{cor}

\begin{proof}If $f(X) = Y$, then $Cl_{X+_{f,g}Y}(X) = X \cup f(X) = X \cup Y$, which implies that $X$ is dense in $X+_{f,g}Y$. If $f(X) = Y_{1} \subsetneq Y$, then $Cl_{X+_{f,g}Y}(X) = X \cup f(X) = X \cup Y_{1} \subsetneq X \cup Y$, which implies that $X$ is not dense in $X+_{f,g}Y$.
\end{proof}

Analogously, $Y$ is dense in $X+_{f,g}Y$ if and only if $g(Y) = X$.

\begin{prop}$Y$ is closed in $X+_{f}Y$.
\end{prop}

\begin{proof}We have that $Y\cap X = \emptyset \in Closed(X), \ Y\cap Y = Y \in Closed(Y)$ and $f(Y\cap X) = f(\emptyset) = \emptyset \subseteq Y$. Thus, $Y \in Closed(X+_{f}Y)$.
\end{proof}

\begin{prop}The maps $id_{X}:X \rightarrow X+_{f,g}Y$ and $id_{Y}:Y \rightarrow X+_{f,g}Y$ are embeddings.
\end{prop}

\begin{proof} Let $F \in Closed(X+_{f,g}Y)$. Then $F\cap X \in Closed(X)$. However, $F \cap X = id_{X}^{-1}(F)$. So $id_{X}$ is continuous. Let $F \in Closed(X)$. We have that $Cl_{X+_{f,g}Y}(F) = F \cup f(F)$ and $(F \cup f(F))\cap X = F$, which implies that $F$ is closed in $X$ as a subspace of $X+_{f,g}Y$. Thus, $id_{X}$ is an embedding.

Analogously $id_{Y}:Y \rightarrow X+_{f,g}Y$ is an embedding.
\end{proof}

\begin{prop}\label{bomcomp}Let $Z$ be a topological space such that $Z = X\dot{\cup} Y$ and $X$ is open. We define $f: Closed(X) \rightarrow Closed(Y)$ as $f(A) = Cl_{Z}(A) \cap Y$ and $g: Closed(Y) \rightarrow Closed(X)$ as $g(B) = Cl_{Z}(B) \cap X$. So $Z$ and $X+_{f,g}Y$ have the same topology.
\end{prop}

\begin{proof}Let $A,A' \in Closed(X)$. So $f(A\cup A') = Cl_{Z}(A\cup A') \cap Y =$ $\\ (Cl_{Z}(A)\cup Cl_{Z}(A')) \cap Y = (Cl_{Z}(A)\cap Y)\cup (Cl_{Z}(A') \cap Y) = f(A)\cup f(A')$ and $f(\emptyset) = Cl_{Z}(\emptyset) \cap Y = \emptyset \cap Y = \emptyset$. So $f$ is admissible. Analogously $g$ is admissible. We have also that $g(f(A)) = Cl_{Z}(Cl_{Z}(A)\cap Y)\cap X \subseteq  Cl_{Z}(Cl_{Z}(A))\cap X  = Cl_{Z}(A)\cap X = Cl_{X}(A) = A$. Analogously, for $B\in Closed(Y)$, $f(g(B)) \subseteq B$. So $f$ and $g$ are an admissible pair.

Let $A \in Closed (Z)$. We have that $A\cap X \in Closed (X), \ A\cap Y \in Closed (Y)$, $f(A \cap X) = Cl_{Z}(A\cap X) \cap Y \subseteq Cl_{Z}(A\cap X) \subseteq Cl_{Z}(A) = A$ and $g(A \cap Y) = Cl_{Z}(A\cap Y) \cap X \subseteq Cl_{Z}(A\cap Y) \subseteq Cl_{Z}(A) = A$. So $A \in Closed(X+_{f,g}Y)$. Let $A \in Closed(X+_{f,g}Y)$. We have that $A\cap X \in Closed(X)$, which implies that $Cl_{X}(A\cap X) = A\cap X \subseteq A$. But $Cl_{X}(A\cap X) = Cl_{Z}(A\cap X)\cap X$, which implies that $Cl_{Z}(A\cap X)\cap X \subseteq A$. For the other hand, we have that $f(A\cap X) \subseteq A$. But $f(A \cap X) = Cl_{Z}(A\cap X) \cap Y$, which implies that $Cl_{Z}(A\cap X) \cap Y \subseteq A$. So $Cl_{Z}(A\cap X) \subseteq A$. We also have that $A\cap Y \in Closed(Y)$ and $g(A\cap Y) \subseteq A$, which implies that $Cl_{Z}(A\cap Y) \subseteq A$. But $A = (A\cap X)\cup(A \cap Y)$, which implies that $Cl_{Z}(A) = Cl_{Z}(A\cap X) \cup Cl_{Z}(A\cap Y) \subseteq A$. So $A \in Closed(Z)$.

Thus, $Closed (Z) = Closed(X+_{f,g}Y)$.
\end{proof}

As a simple example, we have:

\begin{prop}Let $X,Y$ be topological spaces. Then $X+_{\emptyset}Y$ is the coproduct of $X$ and $Y$.
\end{prop}

\begin{proof}We have that $X\cup f(X) = X \in Closed(X+_{\emptyset}Y)$. So $X$ and $Y$ are closed, disjoint and $X\cup Y = X+_{\emptyset}Y$, which implies that $X+_{\emptyset}Y$ is the coproduct of $X$ and $Y$.
\end{proof}

\subsection{Topological Properties}

\subsubsection{Separation}

\begin{prop}Let $X,Y$ be topological spaces and $X+_{f,g}Y$ Hausdorff. So $\forall K \subseteq X$ compact, $f(K) = \emptyset$ and $\forall K \subseteq Y$ compact, $g(K) = \emptyset$.
\end{prop}

\begin{proof}Let $K \subseteq X$ be a compact. We have that $Cl_{X+_{f,g}Y}(K) = K \cup f(K)$. Since $X+_{f,g}Y$ is Hausdorff, it follows that $K$ is closed, which means that $K \cup f(K) = K$, which implies that $f(K) = \emptyset$. Analogously, $\forall K \subseteq Y$ compact, $g(K) = \emptyset$.
\end{proof}

\begin{prop}Let $X,Y$ be Hausdorff spaces, with $X$ locally compact. Then $X+_{f}Y$ is Hausdorff if and only if $\forall K \subseteq X$ compact, $f(K) = \emptyset$ and $\forall a,b \in Y, \ \exists A,B \in Closed(X): \ A \cup B = X, \ b \notin f(A)$ and $a \notin f(B)$.
\end{prop}

\begin{proof}($\Rightarrow$) Let $a,b \in Y$. Since $X+_{f}Y$ is a Hausdorff space, $\exists U,V \in Closed(X+_{f}Y): \ U \cup V = X+_{f}Y, \ a \notin V$ and $b \notin U$. Take $A = U \cap X$ and $B = V \cap X$. We have that $A,B \in Closed(X)$ and $A \cup B = (U\cup V) \cap X = X$. Since $U$ and $V$ are  closed, $f(A) = f(U \cap X) \subseteq U$ and $f(B) = f(V \cap X) \subseteq V$. Thus, $a \notin f(B)$ and $b \notin f(A)$. We already saw in this case that $f(K) = \emptyset$ for every compact $K \subseteq X$.

($\Leftarrow$) Let $a,b \in X$. Since $X$ is Hausdorff, there exists $U,V$ open and disjoint neighborhoods of $a$ and $b$. But $X$ is open in $X+_{f}Y$, which implies that $U$ and $V$ are open and disjoint sets of $X+_{f}Y$ that separate $a\in U$ and $b\in V$.

Let $a\in X$ and $b \in Y$. Since $X$ is locally compact, there exists an open neighborhood $U$ of $a$ in $X$ such that $Cl_{X}(U)$ is compact. Since $X$ is open in $X+_{f}Y$, we have that $U$ is an open neighborhood of $a$ in $X+_{f}Y$ and, since $Cl_{X}(U)$ is compact, we have $f(Cl_{X}(U)) = \emptyset$, which implies that $Cl_{X}(U)$ is closed in $X+_{f}Y$. It follows that $U$ and $(X+_{f}Y) - Cl_{X}(U)$ separate $a$ and $b$.

Let $a,b \in Y$. So there exists $A,B \in Closed(X)$ such that  $A \cup B = X$, $a \notin f(B)$ and $b \notin f(A)$. Since, $Y$ is Hausdorff, there exists $C,D \in Closed(Y)$ such that $C\cup D = Y, \ a \notin D$ and $b \notin C$. We have that $A\cup f(A) \cup C$ and $B\cup f(B) \cup D$ are closed sets in $X+_{f}Y$ such that $A\cup f(A)\cup C\cup B\cup f(B) \cup D = (A\cup B) \cup (f(A) \cup f(B) \cup C \cup D) = X \cup Y = X+_{f}Y, \ a \notin B \cup f(B) \cup D$ and $b \notin A \cup f(A) \cup C$. So $(X+_{f}Y) - (B \cup f(B)\cup D)$ and $(X+_{f}Y) - (A \cup f(A)\cup C)$ are open sets that separate $a$ and $b$.

Thus, $X+_{f}Y$ is Hausdorff.
\end{proof}

\subsubsection{Compactness}

\begin{prop}Let $X,Y$ be topological spaces with $Y$ compact and $f$ an admissible map. Then $X+_{f}Y$ is compact if and only if $\forall A \in Closed(X)$ non compact, $f(A) \neq \emptyset$.
\end{prop}

\begin{proof}($\Rightarrow$) Let $A \in Closed(X)$ be non  compact. Since $X+_{f}Y$ is compact, we have that $A$ is not closed, which implies that $f(A)\neq \emptyset$.

($\Leftarrow$) Let $\mathcal{F}$ be a filter in $X+_{f}Y$. If $\exists K \in \mathcal{F}$ compact, then $\mathcal{F}\cap K = \{A\cap K: A \in \mathcal{F}\}$ is a filter in $K$ which has a cluster point $x$ (because $K$ is compact). Since $K \in \mathcal{F}$, we have that $\mathcal{F}\cap K$ is a basis for $\mathcal{F}$, which implies that $x$ is a cluster point of $\mathcal{F}$.

Let's suppose that $\nexists K \in \mathcal{F}: K$ is compact. Let $S \in \mathcal{F}: Cl_{X+_{f}Y}(S) \subseteq X$. We have that $Cl_{X+_{f}Y}(S) = Cl_{X}(S)$. So $Cl_{X}(S) \in \mathcal{F}$ and  $Cl_{X+_{f}Y}(Cl_{X}(S)) \subseteq X$. Since $Cl_{X+_{f}Y}(Cl_{X}(S)) = Cl_{X}(S) \cup f(Cl_{X}(S))$, we have that $f(Cl_{X}(S)) = \emptyset$, which implies that $Cl_{X}(S)$ is compact, a contradiction. So $\forall A \in \mathcal{F}$, $Cl_{X+_{f}Y}(A) \cap Y \neq \emptyset$.

So we have that the set $\{Cl_{X+_{f}Y}(A)\cap Y\}_{A \in \mathcal{F}}$ has the finite intersection property (if $A_{1},...,A_{n} \in \mathcal{F}$, then $A_{1}\cap...\cap A_{n} \in \mathcal{F}$, which implies that $Cl_{X+_{f}Y}(A_{1} \cap ...\cap A_{n}) \cap Y \neq \emptyset$ and then $Cl_{X+_{f}Y}(A_{1}) \cap ...\cap Cl_{X+_{f}Y}(A_{n}) \cap Y \neq \emptyset$ because $Cl_{X+_{f}Y}(A_{1} \cap ...\cap A_{n}) \subseteq Cl_{X+_{f}Y}(A_{1}) \cap ...\cap Cl_{X+_{f}Y}(A_{n})$). Since $Y$ is compact, $\exists x \in \bigcap_{A \in \mathcal{F}}Cl_{X+_{f}Y}(A)\cap Y$, which implies that $\forall A \in \mathcal{F}$, $x \in Cl_{X+_{f}Y}(A)$. So $x$ is a cluster point of $\mathcal{F}$.

Thus, every filter has a cluster point, which implies that  $X+_{f}Y$ is compact.
\end{proof}

\subsection{Continuous maps between Artin-Wraith glueings}

\begin{defi} Let $X+_{f,g}Y$ and $Z+_{h,j}W$ be spaces and $\psi: X \rightarrow Z$ and $\phi: Y \rightarrow W$ continuous maps. So we define $\psi + \phi: X+_{f,g}Y \rightarrow Z+_{h,j}W$ by $(\psi + \phi)(x) = \psi(x)$ if $x \in X$ and $\phi(x)$ if $x \in Y$. If $G$ is a group, $\psi: G \curvearrowright X$ and $\phi: G \curvearrowright Y$, then we define $\psi+\phi: G \curvearrowright X+_{f,g}Y$ by $(\psi+\phi)(g,x) = \psi(g,x)$ if $x \in X$ and $\phi(g,x)$ if $x \in Y$.
\end{defi}

Our next concern is to decide when those maps are continuous.

\begin{prop}\label{mapacontinuo}Let $X+_{f,g}Y$ and $Z+_{h,j}W$ be topological spaces and $\psi: X \rightarrow Z$ and $\phi: Y \rightarrow W$ continuous maps. Then, the application $\psi + \phi: X+_{f,g}Y \rightarrow Z+_{h,j}W$ is continuous if and only if $\forall A \in Closed(Z)$,  $f(\psi^{-1}(A)) \subseteq \phi^{-1}(h(A))$ and $\forall B \in Closed(W)$,  $g(\phi^{-1}(B)) \subseteq \psi^{-1}(j(B))$.
\end{prop}

\begin{obs}In another words, $\psi+\phi$ is continuous  if and only if we have the diagrams:

$$ \xymatrix{ Closed(X) \ar[r]^{f} & Closed(Y) & & Closed(X)  & Closed(Y) \ar[l]_{g} \\
            Closed(Z) \ar[r]^{h} \ar[u]^{\psi^{-1}} \ar@{}[ur]|{\subseteq} & Closed(W) \ar[u]^{\phi^{-1}} & & Closed(Z)  \ar[u]^{\psi^{-1}} \ar@{}[ur]|{\supseteq} & Closed(W) \ar[l]_{j} \ar[u]^{\phi^{-1}} } $$
\end{obs}

\begin{proof}$(\Leftarrow)$ Let $A$ be a closed set in $Z+_{h,j}W$. We have that $(\psi + \phi)^{-1}(A) =$ $\psi^{-1}(A \cap Z) \cup \phi^{-1}(A \cap W)$. Let's prove that this set is closed, showing that it is equal to its closure. We have that $Cl_{X+_{f,g}Y}(\psi^{-1}(A \cap Z) \cup \phi^{-1}(A \cap W)) = Cl_{X+_{f,g}Y}(\psi^{-1}(A \cap Z)) \cup Cl_{X+_{f,g}Y}(\phi^{-1}(A \cap W))$. But $Cl_{X+_{f,g}Y}(\psi^{-1}(A \cap Z))= \psi^{-1}(A \cap Z) \cup f(\psi^{-1}(A \cap Z))$. We have that $f(\psi^{-1}(A \cap Z)) \subseteq \phi^{-1}(h(A \cap Z))$ by hypothesis and $\phi^{-1}(h(A \cap Z)) \subseteq \phi^{-1}(A \cap W)$, because $A$ is closed in $Z+_{h,j}W$. So $Cl_{X+_{f,g}Y}(\psi^{-1}(A \cap Z)) \subseteq \psi^{-1}(A \cap Z) \cup \phi^{-1}(A \cap W)$.
Using the second diagram we have analogously that $Cl_{X+_{f,g}Y}(\phi^{-1}(A \cap W)) \subseteq \psi^{-1}(A \cap Z) \cup \phi^{-1}(A \cap W)$. So $Cl_{X+_{f,g}Y}(\psi^{-1}(A \cap Z) \cup \phi^{-1}(A \cap W)) \subseteq \psi^{-1}(A \cap Z) \cup \phi^{-1}(A \cap W)$ and follows the equality. Thus, $\psi + \phi$ is continuous.

$(\Rightarrow)$ Let's suppose that $\psi + \phi$ is continuous. Let $A$ be a closed set in $Z$. We have that $A \cup h(A)$ is closed in $Z+_{h,j}W$. By continuity of the map $\psi + \phi$, we have that $(\psi + \phi)^{-1}(A\cup h(A))\in Closed(X+_{f,g}Y)$. But $(\psi + \phi)^{-1}(A\cup h(A)) = \psi^{-1}(A) \cup \phi^{-1}(h(A)) = Cl_{X+_{f,g}Y}(\psi^{-1}(A) \cup \phi^{-1}(h(A))) =$ $Cl_{X+_{f,g}Y}(\psi^{-1}(A)) \cup Cl_{X+_{f,g}Y}(\phi^{-1}(h(A)))$. So we have that $Cl_{X+_{f,g}Y}(\psi^{-1}(A)) \subseteq$ $\psi^{-1}(A) \cup \phi^{-1}(h(A))$. But $Cl_{X+_{f,g}Y}(\psi^{-1}(A)) = \psi^{-1}(A) \cup f(\psi^{-1}(A))$ and $f(\psi^{-1}(A)) \cap \psi^{-1}(A) = \emptyset$, because $\psi^{-1}(A) \subseteq X$. Thus, $f(\psi^{-1}(A)) \subseteq \phi^{-1}(h(A))$, as we wish to proof. Analogously, $\forall B \in Closed(W)$,  $g(\phi^{-1}(B)) \subseteq \psi^{-1}(j(B))$.
\end{proof}

\begin{obs}Something similar appears in \cite{Ne} in the context of locales but with one glueing map instead of two.
\end{obs}

\begin{cor}Let $X+_{f,g}Y$, $X+_{f',g'}Y$ be topological spaces. Then, the map $id: X+_{f,g}Y \rightarrow X+_{f',g'}Y$ is continuous if and only if $\forall A \in Closed (X)$, $f(A) \subseteq f'(A)$ and $\forall B \in Closed(Y)$,  $g(B) \subseteq g'(B)$.
\eod\end{cor}

\subsection{Composition of Artin-Wraith glueings}

\begin{prop}Let $X+_{f}W, \ Y$ and $Z$ be topological spaces and let $\Pi: Closed(Y) \rightarrow Closed (X)$ and $\Sigma: Closed(W) \rightarrow Closed(Z)$ be admissible maps. We define $f_{\Sigma\Pi}: Closed(Y) \rightarrow Closed(Z)$ as $f_{\Sigma\Pi} = \Sigma \circ f \circ \Pi$. Then, $f_{\Sigma\Pi}$ is admissible.
\eod\end{prop}

\begin{prop}Let $X+_{f}W, \ Y$ and $Z$ be topological spaces and consider the admissible maps:
\begin{enumerate}
    \item $\Pi: Closed(Y) \rightarrow Closed (X)$,
    \item $\Sigma: Closed(W) \rightarrow Closed(Z)$,
    \item $\Lambda: Closed(X) \rightarrow Closed(Y)$,
    \item $\Omega: Closed(Z) \rightarrow Closed(W)$.
\end{enumerate}

If $\Omega \circ \Sigma \subseteq id_{Closed(W)}$ (respec. $\supseteq id_{Closed(W)}$ or $=  id_{Closed(W)}$) and $\Pi \circ \Lambda \subseteq id_{Closed (X)}$ (respec. $\supseteq id_{Closed(X)}$ or $=  id_{Closed(X)}$) then $(f_{\Sigma \Pi})_{\Omega \Lambda} \subseteq f$ (respec. $\supseteq f$ or $= f$).
\end{prop}

\begin{proof}We have that $(f_{\Sigma \Pi})_{\Omega \Lambda}(A) =  \Omega \circ f_{\Sigma \Pi}\circ \Lambda(A) = \Omega \circ \Sigma \circ f \circ \Pi \circ \Lambda(A)$. If $\Omega \circ \Sigma \subseteq id_{Closed(W)}$ and $\Pi \circ \Lambda \subseteq id_{Closed (X)}$, then $\Omega \circ \Sigma \circ f \circ \Pi \circ \Lambda(A) \subseteq \Omega \circ \Sigma \circ f(A) \subseteq f(A)$. The other cases are analogous.
\end{proof}

\begin{prop}(Cube Lemma) Let $X_{i}+_{f_{i}}W_{i}$, $Y_{i}$ and $Z_{i}$ be topological spaces, $\Pi_{i}: Closed(Y_{i}) \rightarrow Closed(X_{i})$ and $\Sigma_{i}: Closed(W_{i}) \rightarrow Closed(Z_{i})$ admissible maps, with $i \in \{1,2\}$. Take $f_{i\Sigma_{i}\Pi_{i}}: Closed(Y_{i}) \rightarrow Closed(Z_{i})$ the respective induced maps. If $\mu+\nu: X_{1}+_{f_{1}}W_{1} \rightarrow X_{2}+_{f_{2}}W_{2}$, $\psi: Y_{1} \rightarrow Y_{2}$ and $\phi: Z_{1} \rightarrow Z_{2}$ are continuous maps that form the diagrams:

$$ \xymatrix{   Closed(X_{2}) \ar[r]^{\mu^{-1}} \ar@{}[dr]|{\supseteq} & Closed(X_{1}) & & Closed(W_{2}) \ar[r]^{\nu^{-1}} \ar[d]^{\Sigma_{2}} \ar@{}[dr]|{\supseteq} & Closed(W_{1}) \ar[d]_{\Sigma_{1}} \\
                Closed(Y_{2}) \ar[r]^{\psi^{-1}} \ar[u]^{\Pi_{2}} & Closed(Y_{1}) \ar[u]_{\Pi_{1}} & & Closed(Z_{2}) \ar[r]^{\phi^{-1}} & Closed(Z_{1}) } $$

Then, $\psi+\phi: Y_{1}+_{f_{1\Sigma_{1}\Pi_{1}}}Z_{1}\rightarrow Y_{2}+_{f_{2\Sigma_{2}\Pi_{2}}}Z_{2}$ is continuous.
\end{prop}

\begin{proof}Consider the diagram:

$$ \xymatrix{Closed(X_{1}) \ar[rr]^{f_{1}} & & Closed(W_{1}) \ar[dd]_<<{\Sigma_{1}} \\
            & Closed(X_{2}) \ar[rr]^<<{ \ \ \ \ \ \ \ \ \ \ \ \ \ f_{2}} \ar[lu]^{\mu^{-1}} & & Closed(W_{2}) \ar[lu]_{\nu^{-1}} \ar[dd]_{\Sigma_{2}} \\
             Closed(Y_{1}) \ar[rr]_<<{ \ \ \ \ \ \ \ \ \ \ \ f_{1\Sigma_{1}\Pi_{1}}} \ar[uu]^{\Pi_{1}} & & Closed(Z_{1}) & \\
            & Closed(Y_{2}) \ar[rr]_{f_{2\Sigma_{2}\Pi_{2}}} \ar[uu]^>>{\Pi_{2}} \ar[lu]^{\psi^{-1}} & & Closed(Z_{2}) \ar[lu]_{\phi^{-1}} &} $$

We have that $f_{1\Sigma_{1}\Pi_{1}}\circ\psi^{-1} = \Sigma_{1}\circ f_{1} \circ \Pi_{1}\circ\psi^{-1} \subseteq \Sigma_{1}\circ f_{1}\circ \mu^{-1} \circ \Pi_{2} \subseteq \Sigma_{1}\circ \nu^{-1} \circ f_{2} \circ \Pi_{2} \subseteq \phi^{-1} \circ \Sigma_{2} \circ f_{2} \circ \Pi_{2} = \phi^{-1} \circ f_{2\Sigma_{2}\Pi_{2}}$. Thus, $\psi+\phi$ is continuous.
\end{proof}

\begin{cor}\label{action} Let $G_{1},G_{2}$ be groups, $X+_{f}W$, $Y$ and $Z$ topological spaces, $\alpha: G_{1} \rightarrow G_{2}$ a group homomorphism and $\Pi: Closed(Y) \rightarrow Closed(X)$ and $\Sigma: Closed(W) \rightarrow Closed(Z)$ admissible maps. Take the induced map $f_{\Sigma\Pi}: Closed(Y) \rightarrow Closed(Z)$. If $\mu+\nu: G_{2} \curvearrowright X+_{f}W$, $\psi: G_{1} \curvearrowright Y$ and $\phi: G_{1} \curvearrowright Z$ are actions by homeomorphisms such that form the following diagrams for each $g \in G_{1}$:

$$ \xymatrix{   Closed(X) \ar[r]^*+{\labelstyle \ \ \mu(\alpha(g),\_)^{-1}} \ar@{}[dr]|{\supseteq} & Closed(X) & & Closed(W) \ar[r]^*+{\labelstyle \ \ \nu(\alpha(g),\_)^{-1}} \ar[d]^{\Sigma} \ar@{}[dr]|{\supseteq} & Closed(W) \ar[d]_{\Sigma} \\
                Closed(Y) \ar[r]^{ \ \psi(g,\_)^{-1}} \ar[u]^{\Pi} & Closed(Y) \ar[u]_{\Pi} & & Closed(Z) \ar[r]^{ \ \phi(g,\_)^{-1}} & Closed(Z) } $$

Then, $\psi+\phi: G_{1} \curvearrowright Y+_{f_{\Sigma \Pi}}Z$ is an action by homeomorphisms.
\end{cor}

\begin{proof}It follows by the last proposition that $\forall g \in G_{1}, \ (\psi+\phi)(g,\_)$ is continuous (and then a homeomorphism, since its inverse is $(\psi+\phi)(g^{-1},\_)$, which is also continuous). Thus, $\psi+\phi: G_{1} \curvearrowright Y+_{f_{\Sigma \Pi}}Z$ is an action by homeomorphisms.
\end{proof}

And a useful proposition about separation:

\begin{prop}Let $X+_{f}W, \ Y$ and $Z$ be Hausdorff topological spaces such that $Y$ and $X$ are locally compact and $\Pi: Closed(Y) \rightarrow Closed (X)$, $\Sigma: \ Closed(W) \ \rightarrow \ Closed(Z)$, $\ \Lambda: \ Closed(X) \ \rightarrow \ Closed(Y)$ and $\Omega: Closed(Z) \rightarrow Closed(W)$ are admissible maps. If $\forall C \subseteq Y$ compact, $\Pi(C)$ is compact, $\{\{z\}: z \in Z\} \subseteq Im \ \Sigma$, $\forall w \in W$, $w \in \Omega \circ \Sigma(\{w\})$, $\Lambda(X) = Y$ and $(f_{\Sigma\Pi})_{\Omega\Lambda} = f$, then $Y+_{f_{\Sigma\Pi}}Z$ is Hausdorff.
\end{prop}

\begin{proof}Let $C \subseteq Y$ be a compact. We have that $\Pi(C)$ is compact, which implies that $\Sigma \circ f \circ \Pi(C) = \Sigma (\emptyset) = \emptyset$, because $X+_{f}W$ is Hausdorff.

Let $a,b \in Z$. There exists $C,D \in Closed(W): \ \Sigma (C) = \{a\}$ and $\Sigma(D) = \{b\}$. Take $c \in C$ and $d \in D$. Since $X+_{f}W$ is Hausdorff, there exists $A,B \in Closed(X): \ A \cup B = X, \ d \notin f(A)$ and $c \notin f(B)$. We have that $\Lambda(A), \Lambda(B) \in Closed(Y)$ and $\Lambda(A)\cup \Lambda(B) = \Lambda(A\cup B) = \Lambda(X) = Y$. If $a \in f_{\Sigma\Pi}(\Lambda(B))$, then $\Omega(\{a\}) \subseteq \Omega \circ f_{\Sigma\Pi}(\Lambda(B)) = f(B)$. But $\Omega(\{a\}) = \Omega \circ \Sigma(C) \supseteq \Omega \circ \Sigma(\{c\}) \supseteq \{c\}$, which implies that $c \in f(B)$, a contradiction. So $a \notin f_{\Sigma\Pi}(\Lambda(B))$ and, analogously, $b \notin f_{\Sigma\Pi}(\Lambda(A))$.

Thus, $Y+_{f_{\Sigma\Pi}}Z$ is Hausdorff.
\end{proof}

\begin{cor}\label{functorausdorff}Let $X+_{f}W$ and $Y$ be Hausdorff topological spaces such that $Y$ and $X$ are locally compact and $\Pi: Closed(Y) \rightarrow Closed (X)$ and  $\Lambda: Closed(X) \rightarrow Closed(Y)$ are admissible maps. If $\forall C \subseteq Y$ compact, $\Pi(C)$ is compact, $\Lambda(X) = Y$ and $(f_{id_{W}\Pi})_{id_{W}\Lambda} = f$, then $Y+_{f_{id_{W}\Pi}}W$ is Hausdorff.
\eod\end{cor}

Let's consider two special cases of composition of sums of spaces that shall be useful:

\subsubsection{Pullbacks}

\begin{defi}Let $X+_{f}W$, $Y$ and $Z$ be topological spaces, $\pi: Y \rightarrow X$ and $\varpi: Z \rightarrow W$ be two applications. Consider the admissible applications $\Pi: Closed(Y) \rightarrow Closed(X)$ and $\Sigma: Closed(W) \rightarrow Closed(Z)$ defined as $\Pi(A) = Cl_{X}(\pi(A))$ and $\Sigma(A) = Cl_{Z}(\varpi^{-1}(A))$.  We define the pullback of $f$ with respect to the maps $\pi$ and $\varpi$ by $f^{\ast}(A) = f_{\Sigma \Pi}(A) = Cl_{Z}(\varpi^{-1}(f(Cl_{X}(\pi(A)))))$.
\end{defi}

\begin{prop}If the maps $\pi$ and $\varpi$ are continuous, then the map $\pi+\varpi: Y+_{f^{\ast}}Z \rightarrow X+_{f}W$ is continuous.
\end{prop}

\begin{proof}If $A \in Closed(X)$, then $f^{\ast}(\pi^{-1}(A)) = \varpi^{-1}(f(Cl_{X}(\pi(\pi^{-1}(A))))) \subseteq \varpi^{-1}(f(Cl_{X}(A))) = \varpi^{-1}(f(A))$. In another words, we have the diagram:

$$ \xymatrix{   Closed(Y) \ar[r]^{f^{\ast}} \ar@{}[rd]|{\subseteq} & Closed(Z) \\
                Closed(X) \ar[r]_{f} \ar[u]^{\pi^{-1}} & Closed(W) \ar[u]_{\varpi^{-1}} } $$

Thus, $(\pi + \varpi): Y+_{f^{\ast}}Z \rightarrow X+_{f}W$ is continuous.
\end{proof}

\begin{prop}Suppose that $\pi$ and $\varpi$ are continuous and let $Y+_{f'}Z$, for some admissible map $f'$, such that $\pi+\varpi: Y+_{f'}Z \rightarrow X+_{f}W$ is continuous. Then $id_{Y}+id_{Z}: Y+_{f'}Z \rightarrow Y+_{f^{\ast}}Z$ is continuous.
\end{prop}

\begin{proof}Since the map $\pi+\varpi$ is continuous, it follows that $\forall B \in Closed(X)$, $f'(\pi^{-1}(B)) \subseteq \varpi^{-1}(f(B))$. We have that $\forall A \in Closed(Y)$, $A \subseteq \pi^{-1}(Cl_{X}(\pi(A)))$, which implies that $f'(A) \subseteq f'(\pi^{-1}(Cl_{X}(\pi(A)))) \subseteq \varpi^{-1}(f(Cl_{X}(\pi(A)))) = f^{\ast}(A)$. Thus $id_{Y}+id_{Z}$ is continuous.
\end{proof}

In another words, if $\pi$ and $\varpi$ are continuous, then $f^{\ast}$ induces the coarsest topology (between the topologies that extend the topologies of $Y$ and $Z$) such that the map $\pi+\varpi$ is continuous.

\begin{prop}(Cube Lemma for Pullbacks) Let $X_{i}+_{f_{i}}W_{i},  Y_{i}, Z_{i}$ be topological spaces, $\pi_{i}: Y_{i} \rightarrow X_{i}$ and $\varpi_{i}: Z_{i} \rightarrow W_{i}$ be two maps and  $f_{i}^{\ast}: Closed(Y_{i}) \rightarrow Closed(Z_{i})$ the respective pullbacks. Let's suppose that $\mu+\nu: X_{1}+_{f_{1}}W_{1} \rightarrow X_{2}+_{f_{2}}W_{2}$, $\psi: Y_{1} \rightarrow Y_{2}$ and $\phi: Z_{1} \rightarrow Z_{2}$ are continuous maps that commute the diagrams:

$$ \xymatrix{   Y_{1} \ar[r]^{\psi} \ar[d]^{\pi_{1}} & Y_{2} \ar[d]_{\pi_{2}} & & Z_{1} \ar[r]^{\phi} \ar[d]^{\varpi_{1}} & Z_{2} \ar[d]_{\varpi_{2}} \\
                X_{1} \ar[r]^{\mu} & X_{2} & & W_{1} \ar[r]^{\nu} & W_{2} } $$

Then $\psi+\phi: Y_{1}+_{f_{1}^{\ast}}Z_{1}\rightarrow Y_{2}+_{f_{2}^{\ast}}Z_{2}$ is continuous.
\end{prop}

\begin{proof}We are showing that we have these diagrams:

$$ \xymatrix{   Closed(X_{2}) \ar[r]^{\mu^{-1}} \ar@{}[dr]|{\supseteq} & Closed(X_{1}) & & Closed(W_{2}) \ar[r]^{\nu^{-1}} \ar[d]_{Cl_{Z_{2}}\circ \varpi^{-1}_{2}} \ar@{}[dr]|{\supseteq} & Closed(W_{1}) \ar[d]^{Cl_{Z_{1}}\circ \varpi^{-1}_{1}} \\
                Closed(Y_{2}) \ar[r]^{\psi^{-1}} \ar[u]^{Cl_{X_{2}} \circ \pi_{2}} & Closed(Y_{1}) \ar[u]_{Cl_{X_{1}} \circ \pi_{1}} & & Closed(Z_{2}) \ar[r]^{\phi^{-1}} & Closed(Z_{1}) } $$

Let $A \subseteq X_{2}$. So $A \subseteq Cl_{X_{2}}(A)$, which implies that $\mu^{-1}(A) \subseteq \mu^{-1}(Cl_{X_{2}}(A))$. Since $\mu^{-1}(Cl_{X_{2}}(A))$ is a closed subset of $X_{1}$, it follows that $Cl_{X_{1}}(\mu^{-1}(A)) \subseteq \mu^{-1}(Cl_{X_{2}}(A))$. Let $B \subseteq Y_{2}$ and $x \in \psi^{-1}(B)$. We have that $\pi_{2}\circ \psi(x) \in \pi_{2}(B)$. But $\pi_{2}\circ \psi(x) = \mu \circ \pi_{1}(x)$, which implies that $x \in \pi_{1}^{-1} \circ \mu^{-1} \circ \pi_{2}(B)$. So $\psi^{-1}(B) \subseteq \pi_{1}^{-1}\circ \mu^{-1} \circ \pi_{2}(B)$. Let now $C \in Closed(Y_{2})$. We have that $Cl_{X_{1}}(\pi_{1}(\psi^{-1}(C))) \subseteq$ $Cl_{X_{1}}(\pi_{1}(\pi_{1}^{-1}(\mu^{-1}(\pi_{2}(C))))) \subseteq$ $Cl_{X_{1}}(\mu^{-1}(\pi_{2}(C))) \subseteq$ $\mu^{-1}(Cl_{X_{2}}(\pi_{2}(C)))$. So we have the first diagram.

Let $A \in Closed(W_{2})$. We have that $Cl_{Z_{1}} \circ \varpi_{1}^{-1} \circ \nu^{-1}(A) = Cl_{Z_{1}} \circ \phi^{-1} \circ \varpi_{2}^{-1}(A) \subseteq \phi^{-1} \circ Cl_{Z_{2}} \circ  \varpi_{2}^{-1}(A)$, since $\phi$ is continuous. So we have the second diagram.

By the Cube Lemma, it follows that $\psi+\phi: Y_{1}+_{f_{1}^{\ast}}Z_{1}\rightarrow Y_{2}+_{f_{2}^{\ast}}Z_{2}$ is continuous.

\end{proof}

\begin{cor}\label{pullbackaction}Let $G_{1}$ and $G_{2}$ be groups, $X+_{f}W, Y$ and $Z$ topological spaces, $\alpha: G_{1}\rightarrow G_{2}$ a homomorphism, $\pi: Y \rightarrow X$ and $\varpi: Z \rightarrow W$ two maps and $f^{\ast}: Closed(Y) \rightarrow Closed(Z)$ the pullback of $f$. Let $\mu+\nu: G_{2} \curvearrowright X+_{f}W$, $\psi: G_{1} \curvearrowright Y$ and $\phi: G_{1} \curvearrowright Z$ be actions by homeomorphisms such that $\pi$ and $\varpi$ are $\alpha$-equivariant. Then, the action $\psi+\phi: G_{1} \curvearrowright Y+_{f^{\ast}}Z$ is by homeomorphisms.
\end{cor}

\begin{proof}Since $\pi$ and $\varpi$ are $\alpha$-equivariant, we have that $\forall g \in G$ the diagrams commute:

$$ \xymatrix{   Y \ar[r]^{\psi(g,\_)} \ar[d]_{\pi} & Y \ar[d]^{\pi} & & Z \ar[r]^{\phi(g,\_)} \ar[d]_{\varpi} & Z \ar[d]^{\varpi} \\
                X \ar[r]_{\mu(\alpha(g),\_)} & X & & W \ar[r]_{\nu(\alpha(g),\_)} & W } $$

By the last proposition it follows that $(\psi+\phi)(g,\_) = \psi(g,\_)+\phi(g,\_)$ is continuous. Thus, $\psi+\phi$ is an action by homeomorphisms.

\end{proof}

There are some properties of the pullback:

\begin{prop}\label{hausdorffpullback}If $X+_{f}W,Y,Z$ are Hausdorff, $\varpi$ is injective and $\pi$ and $\varpi$ are continuous, then $Y+_{f^{\ast}}Z$ is Hausdorff.
\end{prop}

\begin{proof}Let $x,y \in Y$. Since $Y$ is Hausdorff, there exists $U,V$ open sets that separate $x$ and $y$. But $Y$ is open in $Y+_{f^{\ast}}Z$, which implies that $U,V$ are open sets in $Y+_{f^{\ast}}Z$ that separate $x$ and $y$. Let $x \in Y$ and $y \in Z$. Take $U,V$ open sets in $X+_{f}W$ that separate $\pi(x)$ and $\varpi(y)$ (which are different points since $\pi(x) \in X$ and $\varpi(y) \in W$). So $(\pi+\varpi)^{-1}(U)$ and $(\pi+\varpi)^{-1}(V)$ separate $x$ and $y$. Now, let $x,y \in Z$. Since $\varpi$ is injective, we have that $\varpi(x) \neq \varpi(y)$ and, since $X+_{f}W$ is Hausdorff, there exists $U$ and $V$ disjoint open sets in $X+_{f}W$ that separate $\varpi(x)$ and $\varpi(y)$. Hence, $(\pi+\varpi)^{-1}(U)$ and $(\pi+\varpi)^{-1}(V)$ are disjoint open sets in $Y+_{f^{\ast}}Z$ that separate $x$ and $y$. Thus, $Y+_{f^{\ast}}Z$ is Hausdorff.
\end{proof}

\begin{prop}If $\pi$ and $\varpi$ are closed and $\varpi$ is surjective, then the map $\pi+\varpi: Y+_{f^{\ast}}Z \rightarrow X+_{f}W$ is closed.
\end{prop}

\begin{proof}Let $A\in Closed(Y+_{f^{\ast}}Z)$. So $A\cap Y \in Closed(Y)$, $A\cap Z \in Closed(Z)$ and $f^{\ast}(A\cap Y) \subseteq A$. Since $\pi$ and $\varpi$ are closed, we have that $\pi(A \cap Y) = (\pi+\varpi)(A)\cap X \in Closed(X)$, $\varpi(A \cap Z) = (\pi+\varpi)(A)\cap W \in$ $Closed(W)$ and $\varpi(f^{\ast}(A\cap Y)) = (\pi+\varpi)(f^{\ast}(A\cap Y)) \subseteq (\pi+\varpi)(A)$. But we have that $f^{\ast}(A\cap Y) = Cl_{Z}(\varpi^{-1}(f(\pi(A\cap Y))))$, which implies that $\varpi(f^{\ast}(A\cap Y)) \supseteq \varpi(\varpi^{-1}(f(\pi(A\cap Y)))) = f(\pi(A\cap Y)) = f((\pi+\varpi)(A)\cap X)$. Therefore $f((\pi+\varpi)(A)\cap X) \subseteq (\pi+\varpi)(A)$. Thus,  $(\pi+\varpi)(A) \in Closed(X+_{f}W)$ and then $\pi+\varpi$ is closed.
\end{proof}

\begin{prop}\label{compact}If $\pi$ is proper and the spaces $Z$ and $X+_{f}W$ are compact, then $Y+_{f^{\ast}}Z$ is compact.
\end{prop}

\begin{proof}Let $A \in Closed(Y)$ be non compact. Then, $Cl_{X}(\pi(A))$ is not compact (otherwise $\pi^{-1}(Cl_{X}(\pi(A))$ would be compact, since $\pi$ is proper, that contain a closed non compact subspace $A$). So $f(Cl_{X}(\pi(A))) \neq \emptyset$ (since $X+_{f}W$ is compact), which implies that $f^{\ast}(A) = Cl_{Z}(\varpi^{-1}(f(Cl_{X}(\pi(A))))) \neq \emptyset$. Thus, $Y+_{f^{\ast}}Z$ is compact.
\end{proof}

\begin{prop}Let $X+_{f}Y$ be a topological space, $X_{1} \subseteq X$ and $Y_{1} \subseteq Y$. So the subspace topology of $Z = X_{1} \cup Y_{1}$ coincides with the topology of the space $X_{1}+_{f_{1}}Y_{1}$, with $f_{1}: Closed(X_{1}) \rightarrow Closed (Y_{1})$ such that $f_{1}(A) = f(Cl_{X}(A))\cap Y_{1}$ is the pullback of $f$ by the inclusion maps.
\end{prop}

\begin{proof}Let $\iota_{X_{1}}: X_{1} \rightarrow X, \ \iota_{Y_{1}}: Y_{1} \rightarrow Y$ and $\iota_{Z}: Z \rightarrow X+_{f}Y$ be the inclusion maps. We have that the maps $\iota_{X_{1}}+\iota_{Y_{1}}: X_{1}+_{f_{1}}Y_{1} \rightarrow X+_{f}Y$ and $\iota_{Z}: Z \rightarrow X+_{f}Y$ are both continuous. By the universal property of the subspace topology we have that $id_{X_{1}\cup Y_{1}}: X_{1}+_{f_{1}}Y_{1} \rightarrow Z$ is continuous and, by the universal property of the pullback, we have the continuity of the inverse. Thus, both topologies coincide.
\end{proof}

\begin{prop}Let $X+_{f}W$, $Y$, $Z$, $U$ and $V$ be topological spaces, $\pi: Y \rightarrow X$, $\varpi: Z \rightarrow W$, $\rho: U \rightarrow Y$ and $\varrho: V \rightarrow Z$ be four maps. Then $f^{\ast\ast} \subseteq (f^{\ast})^{\ast}$, where $f^{\ast\ast}$ is the pullback $f^{\ast\ast}$ of $f$ by the maps $\pi \circ \rho$ and $\varpi \circ \varrho$ and  $(f^{\ast})^{\ast}$ is the pullback of $f^{\ast}$ by the maps $\rho$ and $\varrho$.
\end{prop}

\begin{proof}Let $A \in Closed (U)$. Then $(f^{\ast})^{\ast}(A) = Cl_{V}(\varrho^{-1}(f^{\ast}(Cl_{Y}(\rho(A))))) = Cl_{V}(\varrho^{-1}(Cl_{Z}(\varpi^{-1}(f(Cl_{X}(\pi(Cl_{Y}(\rho(A))))))))) \supseteq$ $\\ Cl_{V}(\varrho^{-1}(\varpi^{-1}(f(Cl_{X}(\pi(\rho(A)))))))  = Cl_{Z}((\varpi \circ \varrho)^{-1}(f(Cl_{X}(\pi\circ\rho(A)))))$.
\end{proof}

\begin{obs}If $U+_{f^{\ast\ast}}V$ is compact and $U+_{(f^{\ast})^{\ast}}V$ is Hausdorff, then $f^{\ast\ast} = (f^{\ast})^{\ast}$. If $\pi$ is closed, then we also have that $f^{\ast\ast} = (f^{\ast})^{\ast}$.
\end{obs}

\subsubsection{Pushforwards}
\begin{defi}Let $X+_{f}W,Y,Z$ be topological spaces and $\pi: X \rightarrow Y$ and $\varpi: W \rightarrow Z$ continuous maps. Let's consider the admissible maps $\Pi: Closed(Y) \rightarrow Closed(X)$ and $\Sigma: Closed(W) \rightarrow Closed(Z)$ as $\Pi(A) = \pi^{-1}(A)$ and $\Sigma(A) = Cl_{Z}(\varpi(A))$. We define the pushforward of $f$ by the maps $\pi$ and $\varpi$ by $f_{\ast}(A) = f_{\Sigma \Pi}(A) = Cl_{Z}(\varpi(f(\pi^{-1}(A))))$.
\end{defi}

\begin{prop}$\pi+\varpi: X+_{f}W \rightarrow Y+_{f_{\ast}}Z$ is continuous.
\end{prop}

\begin{proof}If $A \in Closed(Y)$, then $\varpi^{-1}(f_{\ast}(A)) = \varpi^{-1}(Cl_{Z}(\varpi(f(\pi^{-1}(A))))) \supseteq \varpi^{-1}(\varpi(f(\pi^{-1}(A)))) \supseteq f(\pi^{-1}(A))$, that is, we have the diagram:

$$ \xymatrix{   Closed(X) \ar[r]^{f} \ar@{}[dr]|{\subseteq} & Closed(W) \\
                Closed(Y) \ar[r]_{f_{\ast}} \ar[u]^{\pi^{-1}} & Closed(Z) \ar[u]_{\varpi^{-1}} } $$

Thus, $\pi + \varpi: X+_{f}W \rightarrow Y+_{f_{\ast}}Z$ is continuous.
\end{proof}

\begin{prop}If $\varpi$ is closed and injective, then the diagram commutes:

$$ \xymatrix{   Closed(X) \ar[r]^{f} & Closed(W) \\
                Closed(Y) \ar[r]_{f_{\ast}} \ar[u]^{\pi^{-1}} & Closed(Z) \ar[u]_{\varpi^{-1}} } $$ 
\eod \end{prop}

\begin{prop}\label{propriedadeuniversalpushforward} Let $Y+_{f'}Z$ for some choice of $f'$ such that the map $\pi+\varpi: X+_{f}W \rightarrow Y+_{f'}Z$ is continuous. Then the map $id_{Y}+id_{Z}: Y+_{f_{\ast}}Z \rightarrow Y+_{f'}Z$ is continuous.
\end{prop}

\begin{proof}Since the map $\pi+\varpi$ is continuous, we have that $\forall B \in Closed(Y)$, $f(\pi^{-1}(B)) \subseteq \varpi^{-1}(f'(B))$. We have that $f_{\ast}(B) = Cl_{Z}(\varpi(f(\pi^{-1}(B)))) \subseteq Cl_{Z}(\varpi(\varpi^{-1}(f'(B)))) \subseteq Cl_{Z}(f'(B)) = f'(B)$. Thus we have that the map $id_{Y}+id_{Z}: Y+_{f_{\ast}}Z \rightarrow Y+_{f'}Z$ is continuous.
\end{proof}

In other words, $f_{\ast}$ induces the finer topology (between the topologies that extend the topologies of $Y$ and $Z$) such that the map $\pi+\varpi$ is continuous.

\begin{prop}(Cube Lemma for Pushforwards) Let $X_{i}+_{f_{i}}W_{i}$, $Y_{i}$ and $Z_{i}$ be topological spaces, $\pi_{i}: X_{i} \rightarrow Y_{i}$ and $\varpi_{i}: W_{i} \rightarrow Z_{i}$ continuous maps and $f_{i \ast}: Closed(Y_{i}) \rightarrow Closed(Z_{i})$ the respective pushforwards. If the maps $\mu+\nu: X_{1}+_{f_{1}}W_{1} \rightarrow X_{2}+_{f_{2}}W_{2}$, $\psi: Y_{1} \rightarrow Y_{2}$ and $\phi: Z_{1} \rightarrow Z_{2}$ are continuous and commute the diagrams:

$$ \xymatrix{   X_{1} \ar[r]^{\mu} \ar[d]^{\pi_{1}} & X_{2} \ar[d]_{\pi_{2}} & & W_{1} \ar[r]^{\nu} \ar[d]^{\varpi_{1}} & W_{2} \ar[d]_{\varpi_{2}} \\
                Y_{1} \ar[r]^{\psi} & Y_{2} & & Z_{1} \ar[r]^{\phi} & Z_{2} } $$

Then, $\psi+\phi: Y_{1}+_{f_{1\ast}}Z_{1}\rightarrow Y_{2}+_{f_{2\ast}}Z_{2}$ is continuous.
\end{prop}

\begin{proof}Let $A \subseteq Z_{2}$. Since $\phi$ is continuous, $Cl_{Z_{1}}\phi^{-1}(A) \subseteq \phi^{-1}(Cl_{Z_{2}}(A))$. Let $B \subseteq W_{2}$ and $x \in \varpi_{1}(\nu^{-1}(B))$. So $\phi(x) \in \phi(\varpi_{1}(\nu^{-1}(B))) = \varpi_{2}(\nu(\nu^{-1}(B))) \subseteq \varpi_{2}(B)$, which implies that $x \in \phi^{-1}(\varpi_{2}(B))$. Hence $\varpi_{1}(\nu^{-1}(B)) \subseteq \phi^{-1}(\varpi_{2}(B))$. Let $C \in Closed(W_{2})$. We have that $Cl_{Z_{1}}(\varpi_{1}(\nu^{-1}(C))) \subseteq Cl_{Z_{1}}(\phi^{-1}(\varpi_{2}(C))) \subseteq \phi^{-1}(Cl_{Z_{2}}(\varpi_{2}(C)))$. So we have the diagrams (the first one is immediate from the hypothesis):

$$ \xymatrix{   Closed(X_{2}) \ar[r]^{\mu^{-1}} \ar@{}[dr]|{\circlearrowleft} & Closed(X_{1}) & & Closed(W_{2}) \ar[r]^{\nu^{-1}} \ar[d]_{Cl_{Z_{2}} \circ \varpi_{2}} \ar@{}[dr]|{\supseteq} & Closed(W_{1}) \ar[d]^{Cl_{Z_{1}} \circ \varpi_{1}} \\
                Closed(Y_{2}) \ar[r]^{\psi^{-1}} \ar[u]_{\pi^{-1}_{2}} & Closed(Y_{1}) \ar[u]^{\pi^{-1}_{1}} & & Closed(Z_{2}) \ar[r]^{\phi^{-1}} & Closed(Z_{1}) } $$

By the Cube Lemma, it follows that $\psi+\phi: Y_{1}+_{f_{1\ast}}Z_{1} \rightarrow Y_{2}+_{f_{2\ast}}Z_{2}$ is continuous.
\end{proof}

And a property of the pushforward:

\begin{prop}\label{pushforwardcompact}If $X+_{f}W$ and $Z$ are compact and $\pi$ is surjective, then $Y+_{f_{\ast}}Z$ is compact.
\end{prop}

\begin{proof}Let $A \in Closed(Y)$ be non compact. Then, $\pi^{-1}(A)$ is not compact (otherwise $A = \pi(\pi^{-1}(A))$ would be compact). So $f(\pi^{-1}(A)) \neq \emptyset$, which implies that $f_{\ast}(A) = Cl_{Z}(\varpi(f(\pi^{-1}(A)))) \neq \emptyset$. Thus, $Y+_{f_{\ast}}Z$ is compact.
\end{proof}

The next two propositions relate the notions of pullback and pushforward.

\begin{prop}Let $X+_{f}W$, $Y$ and $Z$ be topological spaces and two continuous maps $\pi: Y \rightarrow X$ and $\varpi: Z \rightarrow W$. If $\pi$ and $\varpi$ are surjective, then $(f^{\ast})_{\ast} = f$. Otherwise $(f^{\ast})_{\ast} \subseteq f$.
\end{prop}

\begin{proof}Suppose that $\pi$ and $\varpi$ are surjective. Let $A \in Closed(X)$. Then $(f^{\ast})_{\ast}(A) = Cl_{W}(\varpi \circ \varpi^{-1}(f(Cl_{X}(\pi \circ \pi^{-1}(A))))) = Cl_{W}(f(Cl_{X}(A))) =$  $\\ Cl_{W}(f(A)) = f(A)$. The other statement is analogous.
\end{proof}

\begin{prop}Let $X+_{f}W$, $Y$ and $Z$ be topological spaces and two continuous maps $\pi: X \rightarrow Y$ and $\varpi: W \rightarrow Z$. Then $(f_{\ast})^{\ast} \supseteq f$. If $\pi$ and $\varpi$ are closed and injective then $(f_{\ast})^{\ast} = f$.
\end{prop}

\begin{proof}Let $A$ be a closed subset of $X$. We have that $(f_{\ast})^{\ast}(A) =$ $\\ \varpi^{-1}(Cl_{Z}(\varpi(f(\pi^{-1} (Cl_{Y}(\pi(A))))))) \supseteq \varpi^{-1}(\varpi(f(\pi^{-1} (\pi(A))))) \supseteq  f(A)$. The other statement is analogous.
\end{proof}

\subsection{Limits}

\begin{defi}Let $X$ be a locally compact space. We define $Sum(X)$ as the category whose objects are Hausdorff spaces of the form $X+_{f}Y$ and morphisms are continuous maps of the form $id+\phi: X+_{f_{1}}Y_{1} \rightarrow X+_{f_{2}}Y_{2}$. Let $Sum_{0}$ be the full subcategory of $Sum(X)$ whose objects are the spaces where $X$ is dense.
\end{defi}

In this section we are going to construct the limits of both categories. The arguments that are used in this section are straightforward but we put them here for the sake of completeness.

\begin{prop}The one point compactification $X+_{f_{\infty}}\{\infty\}$ is the terminal object in $Sum(X)$.
\end{prop}

\begin{proof}Let $X+_{f}Y \in Sum(X)$. It is clear that if there exists a morphism $id+\phi: X+_{f}Y \rightarrow X+_{f_{\infty}}\{\infty\}$, it must be unique ($\phi$ must be the constant map). Let's check that such map is continuous (and then a morphism). Let $F \in Closed(X)$. If $F$ is compact, then $f(F) = f_{\infty}(F) = \emptyset$, which implies that $f \circ id^{-1}(F) = \emptyset = \phi^{-1} \circ f_{\infty}(F)$. If $F$ is not compact, then $f_{\infty}(F) = \{\infty\}$, which implies that $f \circ id^{-1}(F) = f(F) \subseteq Y = \phi^{-1}(\infty) = \phi^{-1} \circ f_{\infty}(F)$. Thus, $id+\phi$ is continuous and then $X+_{f_{\infty}}\{\infty\}$ is the terminal object in $Sum(X)$.
\end{proof}

\begin{prop}Let $\mathcal{C}$ be a category. We define a new category $\hat{\mathcal{C}}$ whose objects are the same as $\mathcal{C}$ and one new object $\infty$ and the morphisms are the same as $\mathcal{C}$ and for each $c \in \hat{\mathcal{C}}$, a new morphism $e_{c}: c \rightarrow \infty$. For morphisms in $\mathcal{C}$, the new composition is the same. For a morphism $\alpha: c_{1} \rightarrow c_{2}$ of $\mathcal{C}$, we define $e_{c_{2}} \circ \alpha = e_{c_{1}}$. And, finally, we define, for $c \in \mathcal{C}, \ e_{\infty} \circ e_{c} = e_{c}$. This becomes actually a category.
\end{prop}

\begin{proof}We have that the identity of an object in $\mathcal{C}$ continuous to be its identity on the new category and $id_{\infty} = e_{\infty}$. Let $\alpha_{i}: c_{i} \rightarrow c_{i+1}$ be morphisms in $\mathcal{C}$. We have that $(\alpha_{3} \circ \alpha_{2}) \circ \alpha_{1} = \alpha_{3} \circ (\alpha_{2} \circ \alpha_{1})$, since this compositions are just the same as in $\mathcal{C}, \ (e_{c_{3}}\circ \alpha_{2}) \circ \alpha_{1} = e_{c_{2}} \circ \alpha_{1} = e_{c_{1}} = e_{c_{3}} \circ (\alpha_{2} \circ \alpha_{1}), \ (e_{\infty} \circ e_{c_{2}}) \circ \alpha_{1} = e_{c_{2}} \circ \alpha_{1} = e_{c_{1}} = e _{\infty} \circ e_{c_{1}} = e_{\infty} \circ (e_{c_{2}} \circ \alpha_{1}), \ (e_{\infty} \circ e_{\infty}) \circ e_{c_{1}} = e_{\infty} \circ e_{c_{1}} = e_{\infty} \circ (e_{\infty} \circ e_{c_{1}})$ and $(e_{\infty} \circ e_{\infty}) \circ e_{\infty} = e_{\infty} = e_{\infty} \circ (e_{\infty} \circ e_{\infty})$. Thus, the composition is associative, which implies that $\hat{\mathcal{C}}$ is a category.
\end{proof}

Let $F: \mathcal{C} \rightarrow Sum(X)$ be a covariant functor, where $\mathcal{C}$ is a small category. We denote for every $c \in \mathcal{C}$, $F(c) = X+_{f_{c}}Y_{c}$.

\begin{prop}There exists only one extension $\hat{F}: \hat{\mathcal{C}} \rightarrow Sum(X)$ of $F$ such that $\hat{F}(\infty) = X+_{f_{\infty}}\{\infty\}$.
\end{prop}

\begin{proof}It is clear that, if such extension exists, it must be unique, since it is already defined on the objects, on the morphisms in $\mathcal{C}$ and for the morphisms $e_{c}: c \rightarrow \infty$, it must be the unique morphism of the form $\hat{F}(c) \rightarrow \infty$. Let's check that $\hat{F}$ is actually a functor. If $c\in \mathcal{C}$, then $\hat{F}(id_{c}) = F(id_{c}) = id_{F(c)}$ and $\hat{F}(id_{\infty}) = id_{X+_{f_{\infty}}\{\infty\}}$, by definition of $\hat{F}$. Let $\alpha_{i}: \ c_{i} \ \rightarrow \ c_{i+1}$ and $e_{c_{i}}: c_{i} \rightarrow \infty$ be morphisms in $\hat{\mathcal{C}}$, with $c_{i} \in \mathcal{C}$ (there is no morphism of the form $\gamma: \infty \rightarrow c_{i}$ and the only morphism $\gamma: \infty \rightarrow \infty$ is the identity). We have that $\hat{F}(\alpha_{2}\circ \alpha_{1}) = F(\alpha_{2}\circ \alpha_{1}) = F(\alpha_{2}) \circ F(\alpha_{1}) = \hat{F}(\alpha_{2}) \circ \hat{F}(\alpha_{1})$, and $\hat{F}(e_{c_{2}} \circ \alpha_{1}) = \hat{F}(e_{c_{1}}) = \hat{F}(e_{c_{2}}) \circ \hat{F}(\alpha_{1})$, since both are the unique morphism $F(c_{1}) \rightarrow X+_{f_{\infty}}\{\infty\}$. Thus, $\hat{F}$ is a functor.
\end{proof}

Let $\iota_{c}: X \rightarrow X+_{f_{c}}Y_{c}$ be the inclusion map. It is clear that $(X,\{\iota_{c}\}_{c\in \hat{\mathcal{C}}})$ is a cone of the functor $\tilde{F}: \hat{\mathcal{C}} \rightarrow Top$  that does the same as $\hat{F}$. So it induces the diagonal map $\Delta: X \rightarrow \lim\limits_{\longleftarrow} \tilde{F}$.

\begin{prop}$\Delta$ is an open embedding.
\end{prop}

\begin{proof}Let $\pi_{c}: \lim\limits_{\longleftarrow} \tilde{F} \rightarrow X+_{f_{c}}Y_{c}$ be the projection maps and let the equivalence relation $x\sim y$, for $x,y \in \lim\limits_{\longleftarrow} \tilde{F}$ if $x = y$ or $\pi_{\infty}(x) = \pi_{\infty}(y) \in X$. In the case that $x\sim y$, we have that $\forall c \in \hat{\mathcal{C}}, \ \hat{F}(e_{c}) \circ \pi_{c}(x) = \pi_{\infty}(x) = \pi_{\infty}(y) = \hat{F}(e_{c}) \circ \pi_{c}(y)$, which implies that $\pi_{c}(x) = \pi_{c}(y)$, since $\hat{F}(e_{c})|_{X}$ is injective (in a fact it is the identity in $X$). So the outer triangle of the diagram commutes $\forall c,c' \in \hat{\mathcal{C}}$ and $\alpha: c \rightarrow c'$ morphism:

$$ \xymatrix{   & \lim\limits_{\longleftarrow} \tilde{F} \ar[ldd]_{\pi_{c}} \ar[rdd]^{\pi_{c'}} \ar[d]^{\omega} &  \\
                  & (\lim\limits_{\longleftarrow} \tilde{F})/ \!\! \sim \ar[ld]^{\omega_{c}} \ar[rd]_{\omega_{c'}} & \\ \tilde{F}(c) \ar[rr]^{\alpha} & & \tilde{F}(c') } $$

Where $\omega: \lim\limits_{\longleftarrow} \tilde{F} \rightarrow (\lim\limits_{\longleftarrow} \tilde{F})/\!\sim$ is the quotient map and $\omega_{c}$ and $\omega_{c'}$ are the maps that commutes the upper triangles (they are continuous because of the quotient topology). So the whole diagram commutes, which implies, by the universal property of the limit, that $\omega$ is an homeomorphism, which implies that $\sim$ is trivial and then $\pi_{\infty}|_{\pi_{\infty}^{-1}(X)}$ is injective.

Since $\pi_{\infty} \circ \Delta = \iota_{\infty}$ and $\iota_{\infty}$ is injective, it follows that $\Delta$ is injective. Let $U \subseteq X$ be an open set. We have that $\pi_{\infty}^{-1}(U)$ is open in $\lim\limits_{\longleftarrow} \tilde{F}$. But $\pi_{\infty}^{-1}(U) = \Delta(U)$, since $\pi_{\infty} \circ \Delta(U) = U$ and $\pi_{\infty}|_{\pi_{\infty}^{-1}(X)}$ is injective. So $\Delta(U)$ is open. Thus, $\Delta$ is open.
\end{proof}

So $\lim\limits_{\longleftarrow} \tilde{F} \cong X+_{f}Y$ for some topological space $Y$ and an admissible map $f$. Thus, $\lim\limits_{\longleftarrow} \tilde{F}$ can be seen as an object of $Sum(X)$.

\begin{prop}\label{limitesegundotermo}Let $\acute{F}: \hat{C} \rightarrow Top$ defined by $\acute{F}(c) = Y_{c}$ and, for $\alpha: c \rightarrow d$ a morphism, $\acute{F}(\alpha) = \tilde{F}(\alpha)|_{Y_{c}}: Y_{c} \rightarrow Y_{d}$. Then, $Y \cong \lim\limits_{\longleftarrow} \acute{F}$.
\end{prop}

\begin{proof}Let, for $c \in \hat{\mathcal{C}}, \ \nu_{c}: \acute{F}(c) \rightarrow \tilde{F}(c)$ be the inclusion map. By the definition of $\acute{F}$, we have that $\{v_{c}\}_{c\in \hat{\mathcal{C}}}$ is a natural transformation, which implies that it induces a continuous map $\nu: \lim\limits_{\longleftarrow} \acute{F} \rightarrow \lim\limits_{\longleftarrow} \tilde{F}$.

Let $x,y \in \lim\limits_{\longleftarrow} \acute{F}$ such that $\nu(x) = \nu(y)$. Then, $\forall c \in \hat{\mathcal{C}}, \ \pi_{c} \circ \nu(x) = \pi_{c} \circ \nu(y)$. But $\pi_{c} \circ \nu = \nu_{c} \circ \varpi_{c}$, where $\varpi_{c}: \lim\limits_{\longleftarrow} \acute{F} \rightarrow \acute{F}(c)$ is the projection map. So $\forall c \in \hat{\mathcal{C}}, \ \nu_{c} \circ \varpi_{c}(x) =  \nu_{c} \circ \varpi_{c}(y)$, which implies that $\forall c \in \hat{\mathcal{C}}, \ \varpi_{c}(x) = \varpi_{c}(y)$, since $\nu_{c}$ is injective. So $x = y$ and then $\nu$ is injective.

Let $x \notin Im \ \Delta$. If there exists $c_{0} \in \hat{\mathcal{C}}: \pi_{c_{0}}(x) \in X$, then $\pi_{\infty}(x) =  \pi_{c_{0}}(x) \in X$, which implies that $\forall c \in \hat{\mathcal{C}}, \ \pi_{c}(x) \in X$, contradicting the fact that $x \notin Im \ \Delta$. So $\forall c \in \hat{\mathcal{C}}, \ \pi_{c}(x) \in Y_{c}$. Since $\forall \alpha: c \rightarrow d, \ \acute{F}(\alpha)(\pi_{c}(x)) = \tilde{F}(\alpha)(\pi_{c}(x)) = \pi_{d}(x)$, there exists $y \in \lim\limits_{\longleftarrow} \acute{F}$, such that $\forall c \in \hat{\mathcal{C}}, \ \varpi_{c}(y) = \pi_{c}(x)$. So $\forall c \in \hat{\mathcal{C}}, \ \pi_{c} \circ \nu(y) = \nu_{c} \circ \varpi_{c}(y) = \varpi_{c}(y) = \pi_{c}(x)$, which implies that $\nu(y) = x$ and then $(\lim\limits_{\longleftarrow} \tilde{F}) - Im \ \Delta \subseteq Im \ \nu$. Let $x \in Im \ \nu$ and $y \in \lim\limits_{\longleftarrow} \acute{F}: \nu(y) = x$. Then, $\forall c \in \hat{\mathcal{C}}, \ \pi_{c}(x) = \pi_{c} \circ \nu(y) = \nu_{c} \circ \varpi_{c}(y) = \varpi_{c}(y) \in Y_{c}$, which implies that $x \notin Im \ \Delta$. So $Im \ \nu \subseteq (\lim\limits_{\longleftarrow} \tilde{F}) - Im \ \Delta$ and then $Im \ \nu = (\lim\limits_{\longleftarrow} \tilde{F}) - Im \ \Delta$.

Let $W$ be a topological space and $g: W \rightarrow \lim\limits_{\longleftarrow} \acute{F}$ a map such that $\nu \circ g$ is continuous. Consider the following commutative diagram (for every object $c$ in $\hat{C}$):

$$ \xymatrix{ W \ar[r]^{g} & \lim\limits_{\longleftarrow} \acute{F}  \ar[r]^{\nu} \ar[d]^{\varpi_{c}} & \lim\limits_{\longleftarrow} \tilde{F} \ar[d]^{\pi_{c}} \\
                  & Y_{c} \ar[r]_{\nu_{c}} & X\!+_{f_{c}}\!Y_{c}\\ } $$
                  
We have that for every object $c \in \hat{C}$, $\pi_{c}\circ \nu \circ g = \nu_{c} \circ \varpi_{c} \circ g$ is continuous, which implies that $\varpi_{c} \circ g$ is continuous (because $\nu_{c}$ is an embedding). Then $g$ is continuous, by the universal property of $\lim\limits_{\longleftarrow} \acute{F}$, which implies that $\nu$ is an embedding.

Thus, $\lim\limits_{\longleftarrow} \acute{F} \cong Im \ \nu = (\lim\limits_{\longleftarrow} \tilde{F}) - Im \ \Delta \cong Y$.
\end{proof}

\begin{prop}$\lim\limits_{\longleftarrow} \hat{F}$ exists and $\lim\limits_{\longleftarrow} \tilde{F} \cong \lim\limits_{\longleftarrow} \hat{F}$.
\end{prop}

\begin{proof}We have that $\lim\limits_{\longleftarrow} \hat{F}$ satisfies the limit conditions since it satisfies the conditions in $Top$ with more morphisms.
\end{proof}

\begin{prop}$\lim\limits_{\longleftarrow} F$ exists and $\lim\limits_{\longleftarrow} \hat{F} \cong \lim\limits_{\longleftarrow} F$.
\end{prop}

\begin{proof}Both functors have the same cones because they agree in $\mathcal{C}$ and $\hat{F}(\infty)$ is the terminal object in $Sum(X)$. So they have the same limit.
\end{proof}

Observe that $X$ does not need to be dense in the limit, even when $X$ is dense in $F(c)$, $\forall c \in \mathcal{C}$:

\begin{ex}Consider two copies of the two point compactification space $\\ \R+_{f}\{-\infty,\infty\}$. The product is a Hausdorff compact space of the form $\R+_{g}Y$, with $\# Y = 4$, since $Y \cong \{-\infty,\infty\} \times \{-\infty,\infty\}$. But there is no Hausdorff compactification of $\R$ with four points. Thus, $\R$ is not dense in $\R+_{g}Y$.
\end{ex}

So now let's consider the category $Sum_{0}(X)$.

\begin{prop}Let $I: Sum_{0}(X) \rightarrow Sum(X)$ be the inclusion functor and $J: Sum(X) \rightarrow Sum_{0}(X)$ the functor that sends a space $X+_{f}Y$ to $Cl_{X+_{f}Y}(X)$ and a map $id+\phi: X+_{f_{1}}Y_{1} \rightarrow X+_{f_{2}}Y_{2}$ to its restriction $id+\phi|_{f_{1}(X_{1})}: Cl_{X+_{f_{1}}Y_{1}}(X) \rightarrow Cl_{X+_{f_{2}}Y_{2}}(X)$. Then, $J$ is right adjoint to $I$.
\end{prop}

\begin{proof}Let $X+_{f}Y \in Sum_{0}(X)$, $X+_{g}Z \in Sum(X)$ and an application $id+\phi: X+_{f}Y \rightarrow J(X+_{g}Z) = Cl_{X+_{g}Z}(X)$. If $\iota:  Cl_{X+_{g}Z}(X) \rightarrow  X+_{g}Z$ is the inclusion map, then $\iota \circ (id+\phi): I(X+_{f}Y) = X+_{f}Y \rightarrow X+_{g}Z$ is the only morphism that commutes the diagram:

$$ \xymatrix{ X+_{f}Y \ar[r]^-{id}  \ar[d]^{id+\phi} & J \circ I(X+_{f}Y)  \ar[ld]^{ \ J(\iota\circ(id+\phi))} \\
            J(X+_{g}Z) & } $$

So $X+_{f}Y$ and the map $id+id: X+_{f}Y \rightarrow X+_{f}Y$ form a reflection of $X+_{f}Y$ along $J$. Thus, $I$ is left adjoint to $J$.
 \end{proof}

\begin{prop}Let $F: \mathcal{C} \rightarrow Sum_{0}(X)$ be a functor, $\breve{F}: \mathcal{C} \rightarrow Sum(X)$ a functor that do the same thing as $F$ and $Z = \lim\limits_{\longleftarrow} \breve{F}$. Then,  $\lim\limits_{\longleftarrow} F$ exists and $Cl_{Z}X = \lim\limits_{\longleftarrow} F$.
\end{prop}

\begin{proof}RAPL.
\end{proof}

\begin{prop}\label{densonolimite}Let $F: \mathcal{C} \rightarrow Sum_{0}(X)$ be a functor, where $\mathcal{C}$ is a codirected poset. Then, $X$ is dense in $X+_{f}Y = \lim\limits_{\longleftarrow}\breve{F}$ and $\lim\limits_{\longleftarrow}F \cong \lim\limits_{\longleftarrow}\breve{F}$.
\end{prop}

\begin{proof}Let $y \in Y - f(X)$. Since $Y - f(X)$ is open in $X+_{f}Y$ and the set $\{\pi_{c}^{-1}(U): c \in \mathcal{C}$ and $U$ is open in $F(c)\}$ is a basis for the topology of $X+_{f}Y$ (because $\mathcal{C}$ is codirected), we have that there exists $c \in \mathcal{C}$ and $U$ an open set of $F(c) = X+_{f_{c}}Y_{c}$ such that $y \in \pi^{-1}_{c}(U) \subseteq Y - f(X)$. So $\pi_{c}(y) \in U$. Let $x \in X \cap U$. We have that $x \in \pi_{c}^{-1}(U) \subseteq Y$, a contradiction. So $X \cap U = \emptyset$. But $X \subseteq (X+_{f_{c}}Y_{c}) - U$, a closed set, implies that $X \cup f_{c}(X) = Cl_{X+_{f_{c}}Y_{c}}(X) \subseteq (X+_{f_{c}}Y_{c}) - U$. Since $f_{c}(X) = Y_{c}$, it follows that $X+_{f_{c}}Y_{c} \subseteq (X+_{f_{c}}Y_{c}) - U$ and then $U = \emptyset$, a contradiction, since $\pi_{c}(y) \in U$. Thus, $f(X) = Y$ and then $X$ is dense in $X+_{f}Y$. Since $\lim\limits_{\longleftarrow}F \cong Cl_{\lim\limits_{\longleftarrow}\breve{F}}(X)$, it follows that $\lim\limits_{\longleftarrow}F \cong \lim\limits_{\longleftarrow}\breve{F}$
\end{proof}

\begin{cor}\label{limitesegundotermominimal}Let $F: \mathcal{C} \rightarrow Sum_{0}(X)$ be a functor, where $\mathcal{C}$ is a codirected poset and $X+_{f}Y = \lim\limits_{\longleftarrow}F$. Then, $Y \cong \lim\limits_{\longleftarrow} \acute{F}$.
\eod\end{cor}

\section{Compactifications}

\subsection{Freudenthal compactification}

Let $X$ be a Hausdorff connected and locally connected space. We  construct the Freudenthal compactification of $X$ \cite{Fr} (the universal compactification of $X$ with totally disconnected boundary) using inverse limits. This construction of the Freudenthal compactification is well known (it can be found in section 9.1 of \cite{DK}) but we present it here  using the methods of Artin-Wraith glueings that we presented before.

For $K \subseteq X$ let's define $\pi_{0}^{u}(X-K)$ as the set of unbounded connected components of $X - K$ with the discrete topology (boundedness here means that its closure in $X$ is compact).

\begin{prop}$\forall K \subseteq X$ compact, $\pi_{0}^{u}(X-K)$ is finite.
\end{prop}

\begin{proof}Let $V$ be an open set such that $K \subseteq V$ and $Cl_{X}(V)$ is compact (it exists since $X$ is locally compact). Let $S = \{U \in \pi_{0}^{u}(X-K): U \cap \partial V \neq \emptyset\}$. Since $X$ is connected, $\partial V \neq \emptyset$. Let, $\forall p \in  \partial V, \ V_{p}$ be an open and connected neighborhood of $p$ such that $V_{p} \cap K = \emptyset$. Let $\{V_{p_{1}},...,V_{p_{n}}\}$ be a finite subcover of $\partial V$ ($\partial V$ is compact, since $Cl_{X}V$ is compact). If $U \in S$, then there exists $i \in \{1,...,n\}: V_{p_{i}} \subseteq U$. However, for $U\neq U' \in S$ and $i,i'\in\{1,...,n\}: V_{p_{i}} \subseteq U$ and $V_{p_{i'}} \subseteq U'$, we have that $i \neq i'$ (since $U$ and $U'$ are disjoint). So $S$ is finite.

Let $W \in \pi_{0}^{u}(X-K)-S$. Then, $W \subseteq V \cup (X - Cl_{X}V)$. Since $W$ is connected, we have that $W \subseteq V$ or $W \subseteq X - Cl_{X}V$. Since $V$ is bounded and $W$ is not, it follows that $W \subseteq X - Cl_{X}V$. Since $X-K$ is locally connected, it follows that $W$ is open in $X-K$ and then open in $X$. But $X-W = V \cup \bigcup(\pi_{0}^{u}(X-K)-\{W\})$ is an open set as well, contradicting the fact that $X$ is connected. Thus, $\pi_{0}^{u}(X-K) = S$, which implies that $\pi_{0}^{u}(X-K)$ is finite.
\end{proof}

If $K_{1}, K_{2}$ are two compact subspaces of $X$ with $K_{1} \subseteq K_{2}$, take the map $\psi_{K_{1}K_{2}}: \pi_{0}^{u}(X-K_{2}) \rightarrow \pi_{0}^{u}(X-K_{1})$ defined by $\psi_{K_{1}K_{2}}(U)$ as the connected component of $X-K_{1}$ that contains $U \in \pi_{0}^{u}(X-K_{2})$. We have that, for $K_{1} \subseteq K_{2} \subseteq K_{3} \subseteq X, \ \psi_{K_{1}K_{2}} \circ \psi_{K_{2}K_{3}} = \psi_{K_{1}K_{3}}$. So we are able to define the end space of $X$ as $Ends(X) = \lim\limits_{\longleftarrow} \pi_{0}^{u}(X-K)$.

Let's consider, for $K \subseteq X$ a compact subspace, the space $X+_{f_{K}}\pi_{0}^{u}(X-K)$ with $f_{K}(F) = \{U \in \pi_{0}^{u}(X-K): U \cap F$ is unbounded$\}$.

\begin{prop}$f_{K}$ is admissible.
\end{prop}

\begin{proof}We have that $f_{K}(\emptyset) = \{U \in \pi_{0}^{u}(X-K): U \cap \emptyset$ is unbounded$\} = \emptyset$. Let $F_{1},F_{2} \in Closed(X)$. If $U \in f_{K}(F_{1})$, then $U \cap F_{1}$ is unbounded, which implies that $U \cap (F_{1}\cup F_{2})$ is unbounded and then $U \in f_{K}(F_{1} \cup F_{2})$. Analogously, $f_{K}(F_{2}) \subseteq f_{K}(F_{1}\cup F_{2})$, which implies that $f_{K}(F_{1}) \cup f_{K}(F_{2}) \subseteq f_{K}(F_{1}\cup F_{2})$. If $U \in \pi_{0}^{u}(X-K) - (f_{K}(F_{1}) \cup f_{K}(F_{2}))$, then $U \cap F_{1}$ and $U \cap F_{2}$ are bounded. Hence $(U \cap F_{1}) \cup (U \cap F_{2}) = U \cap (F_{1} \cup F_{2})$ is bounded, which implies that $U \notin f_{K}(F_{1}\cup F_{2})$. So $f_{K}(F_{1}\cup F_{2}) \subseteq f_{K}(F_{1}) \cup f_{K}(F_{2})$, which implies that $f_{K}(F_{1}\cup F_{2}) = f_{K}(F_{1}) \cup f_{K}(F_{2})$. Thus, $f_{K}$ is admissible.
\end{proof}

\begin{prop}$X+_{f_{K}}\pi_{0}^{u}(X-K)$ is compact.
\end{prop}

\begin{proof}Let $F \in Closed(X)$ be non compact. If $f_{K}(F) = \emptyset$, then $\forall U \in \pi_{0}^{u}(X-K), \ U\cap F$ is bounded, which implies that $K \cup \bigcup\limits_{U \in \pi_{0}^{u}(X-K)} (U\cap F)$ is bounded (since $\pi_{0}^{u}(X-K)$ is finite). But $F \subseteq K \cup \bigcup\limits_{U \in \pi_{0}^{u}(X-K)} (U\cap F)$, contradicting the fact that $F$ is not compact (and then unbounded). Thus, $f_{K}(F) \neq \emptyset$, which implies that $X+_{f_{K}}\pi_{0}^{u}(X-K)$ is compact.
\end{proof}

\begin{prop}$X+_{f_{K}}\pi_{0}^{u}(X-K)$ is Hausdorff.
\end{prop}

\begin{proof}Let $F$ be a compact subset of $X$. We have that $\forall U \in \pi_{0}^{u}(X-K), \ U \cap F$ is bounded, which implies that $f_{K}(F) = \emptyset$. Let $U,V \in \pi_{0}^{u}(X-K)$. Since $U$ and $V$ are open in $X$, we have that $X-U, X-V \in Closed(X)$. But $X = (X-U) \cup (X-V)$ and $f_{K}(X-U) = \pi_{0}^{u}(X-K) - \{U\}, \ f_{K}(X-V) = \pi_{0}^{u}(X-K) - \{V\}$, which implies that $U \notin f_{K}(X-U)$ and $V \notin f_{K}(X-V)$. Thus, $X+_{f_{K}}\pi_{0}^{u}(X-K)$ is Hausdorff.
\end{proof}

If $K_{1} \subseteq K_{2} \subseteq X$ are two compact subspaces, we are able to consider the map $id+\psi_{K_{1}K_{2}}: X+_{f_{K_{1}}}\pi_{0}^{u}(X-K_{2}) \rightarrow X+_{f_{K_{1}}}\pi_{0}^{u}(X-K_{1})$.

\begin{prop}$id+\psi_{K_{1}K_{2}}$ is continuous.
\end{prop}

\begin{proof}Let $F \in Closed(X)$ and $U \in \pi_{0}^{u}(X-K_{2})$. If $U \cap F$ is unbounded, then $\psi_{K_{1}K_{2}}(U) \cap F$ is unbounded (since $U \subseteq \psi_{K_{1}K_{2}}(U)$). So if $U \in f_{K_{2}}(F)$, then $\psi_{K_{1}K_{2}}(U) \in f_{K_{1}}(F)$. Hence, $f_{K_{2}}(F) \subseteq \psi_{K_{1}K_{2}}^{-1} \circ f_{K_{1}}(F)$. In another words, we have the diagram:

$$ \xymatrix{ Closed(X) \ar[r]^<<{ \ \ \ \ f_{K_{2}}} & Closed(\pi_{0}^{u}(X-K_{2})) \\
            Closed(X) \ar[r]^<<{ \ \ \ \ f_{K_{1}}} \ar[u]^{id^{-1}} \ar@{}[ur]|{\subseteq \ \ \ \ } & Closed(\pi_{0}^{u}(X-K_{1})) \ar[u]^{\psi_{K_{1}K_{2}}^{-1}} } $$

Thus, $id+\psi_{K_{1}K_{2}}$ is continuous.
\end{proof}

Let $\K$ be the category defined by the poset of compact subspaces of $X$ with the partial order defined by inclusions. Let $\digamma: \K \rightarrow Sum_{0}(X)$ defined by $\digamma(K) = X+_{f_{K}}\pi_{0}^{u}(X-K)$ and, if $K_{1} \subseteq K_{2} \subseteq X$ are compact, $\digamma(K_{1} \subseteq K_{2}) = id+\psi_{K_{1}K_{2}}: X+_{f_{K_{2}}}\pi_{0}^{u}(X-K_{2}) \rightarrow X+_{f_{K_{1}}}\pi_{0}^{u}(X-K_{1})$.

\begin{prop}$\lim\limits_{\longleftarrow}\digamma \cong X+_{f}Ends(X)$, for some admissible map $f$.
\end{prop}

\begin{proof}It is immediate from the \textbf{Propositions \ref{limitesegundotermo}} and \textbf{\ref{densonolimite}}.
\end{proof}

\begin{lema}Let $X+_{g}Z \in Sum_{0}(X)$ be a compact space with $Z$ totally disconnected. For $K \subseteq X$ a compact, we define, for $z_{1},z_{2} \in Z, \ z_{1}\sim_{K}z_{2}$ if $z_{1}$ and $z_{2}$ are in the same connected component in $(X+_{g}Z) - K$ and extend it trivially to $X$. Let's define $Z_{K} = Z/ \!\sim_{K}$, $g_{K}$ an admissible map such that $X+_{g_{K}} Z_{K} = X+_{g}Z/\!\sim_{K}$ (via the identification of $X$ with its classes) and the functor $\digamma_{g}: \K \rightarrow Sum_{0}(X)$ defined by $\digamma_{g}(K) = X+_{g_{K}}Z_{K}$ and $\digamma_{g}(K_{1} \subseteq K_{2}): X+_{g_{K_{2}}}Z_{K_{2}} \rightarrow X+_{g_{K_{1}}}Z_{K_{1}}$ the quotient map. Then, $\lim\limits_{\longleftarrow} \digamma_{g} = X+_{g}Z$.
\end{lema}

\begin{obs}We have that $Z = g(Cl_{X}(X-K)) = \bigcup\limits_{U\in \pi_{0}^{u}(X-K)}g(Cl_{X}(U))$, which implies that every element of $Z$ is in the closure of some element of $\pi_{0}^{u}(X-K)$ and then in a connected component of some element of $\pi_{0}^{u}(X-K)$. Since $\pi_{0}^{u}(X-K)$ is finite and each element of $Z$ must be in a component of an element of $\pi_{0}^{u}(X-K)$ , it follows that $Z_{K}$ is finite. We have also that every connected component is closed, which implies that every class of $Z$ is closed. So $\sim_{K} = \Delta^{2}(X+_{g}Z) \cup \bigcup_{z\in Z} [z]$ is closed, which implies that $X+_{g_{K}}Z_{K}$ is Hausdorff (by \textbf{Alexandrov Theorem}), and then, an element of $Sum_{0}(X)$.
\end{obs}

\begin{proof}Let, for $K \subseteq X$ compact, $id+\eta_{K}: X+_{g}Z \rightarrow X+_{g_{K}}Z_{K}$ be the quotient map. We have that $X+_{g}Z$, together with the family $\{id+\eta_{K}\}_{K\in \K}$, is a cone of $\digamma$. So it induces a continuous map $id+\eta: X+_{g}Z \rightarrow X+_{\tilde{g}}\tilde{Z} = \lim\limits_{\longleftarrow}  \digamma$ (with $\tilde{Z} \cong \lim\limits_{\longleftarrow} Z_{K}$). Since $\K$ is codirected and $\forall K\in \K, \ X$ is dense in $X+_{g_{K}}Z_{K}$, it follows that $X$ is dense in $X+_{\tilde{g}}\tilde{Z}$ and then, the map $id+\eta$ is surjective. Let $x\neq y \in Z$. Since $Z$ is totally disconnected, there exists a clopen set $A$ of $Z$ such that $x \in A$ and $y \in Z-A$. Since $X+_{g}Z$ is normal, there exists $\tilde{A},\tilde{B}$, open sets of $X+_{g}Z$, such that $A \subseteq \tilde{A}, \ Z-A \subseteq \tilde{B}$ and $\tilde{A} \cap \tilde{B} = \emptyset$. Take $K = (X+_{g}Z) - (\tilde{A}\cup \tilde{B})$. We have that $K$ is compact and $K \subseteq X$. Since $(X+_{g}Z) - K = \tilde{A}\cup \tilde{B}$ and $\tilde{A}$ and $\tilde{B}$ are open in $X+_{g}Z$, we have that $\tilde{A}$ and $\tilde{B}$ are clopen in $(X+_{g}Z) - K$. So $x$ and $y$ are not in the same connected component of $(X+_{g}Z) - K$, which implies that $x\nsim_{K}y$ and then $\eta_{K}(x) \neq \eta_{K}(y)$. So $(id+\eta)(x) \neq (id+\eta)(y)$, which implies that $id+\eta$ is injective. Thus, $id+\eta$ is bijective and, since $X+_{g}Z$ is compact and $X+_{\tilde{g}}\tilde{Z}$ is Hausdorff, it is a homeomorphism.
\end{proof}

\begin{prop}\label{universalpropertyends}Let $X+_{g}Z \in Sum_{0}(X)$ be a compact space with $Z$ totally disconnected. Then, there exists only one continuous surjective map of the form $id:\phi: X+_{f}Ends(X) \rightarrow X+_{g}Z$.
\end{prop}

\begin{proof}Let $K \subseteq X$ be a compact. We define $\zeta_{K}: \pi_{0}^{u}(X-K) \rightarrow Z_{K}$ as $\zeta_{K}(U) = [z]$, where $z$ is a point of $Z$ that is in the same connected component of $U$ in $(X+_{g}Z)-K$. By the definition of $\sim_{K}$, the map $\zeta_{K}$ doesn't depend of the choice of $z$, it is continuous because the space $\pi_{0}^{u}(X-K)$ is discrete and it is surjective since every element of $Z$ is in some connected component of an element of $\pi_{0}^{u}(X-K)$.

Let $F\in Closed(X)$ and $U \in f_{K}(F)$. We have that $F\cap U$ is unbounded. Let $z \in Cl_{X+_{g}Z}(F\cap U)\cap Z$ (it exists since $F\cap U$ is unbounded). Since $U$ is connected, $z$ is in the connected component of $U$, which implies that $\zeta_{K}(U) = [z]$, and then, $U \in \zeta_{K}^{-1}([z])$. However, $[z] \in Cl_{X+_{g_{K}}Z_{K}}(F\cap U)\cap Z_{K} = g_{K}(F)$ (since $z \in Cl_{X+_{g}Z}(F\cap U)\cap Z$ and the quotient map is continuous). So $U \subseteq \zeta_{K}^{-1}(g_{K}(F))$ and then  $f_{K}(F) \subseteq \zeta_{K}^{-1}(g_{K}(F))$. In another words, we have the diagram:

$$ \xymatrix{ Closed(X) \ar[r]^<<{ \ \ \ \ f_{K}} & Closed(\pi_{0}^{u}(X-K)) \\
            Closed(X) \ar[r]^<<{ \ \ \ \ \ \ g_{K}} \ar[u]^{id^{-1}} \ar@{}[ur]|{\subseteq} & Closed(Z_{K}) \ar[u]^{\zeta_{K}^{-1}} } $$

So $id+\zeta_{K}: X+_{f_{K}}\pi_{0}^{u}(X-K) \rightarrow X+_{g_{K}}Z_{K}$ is continuous.

Let $K_{1} \subseteq K_{2} \subseteq X$ be compact subspaces. It is clear that the diagram commutes:

$$ \xymatrix{ X \!\! +_{f_{K_{2}}} \! \! \pi_{0}^{u}(X \! - \! K_{2}) \ar[r]_*+{\labelstyle \ \ id + \psi_{K_{1}K_{2}}} \ar[d]^{id+\zeta_{K_{2}}} & X \!\! +_{f_{K_{1}}} \! \! \pi_{0}^{u}(X \!- \! K_{1}) \ar[d]^{id+\zeta_{K_{1}}} \\
            X \!\! +_{g_{K_{2}}} \! \! Z_{K_{2}} \ar[r]^{\digamma_{g}(K_{1} \subseteq K_{2})} & X \!\! +_{g_{K_{1}}} \! \! Z_{K_{1}}  }$$

So it induces a continuous map $id+\zeta: X+_{f}Ends(X) \rightarrow X+_{g}Z$. Since $X$ is dense in $X+_{f}Ends(X)$, it is the only map that extends $id_{X}$ and since $X$ is dense in $X+_{g}Z$, it follows that the map must be surjective.
\end{proof}

\begin{cor}\label{universalends}Let $X+_{g}Z \in Sum(X)$ be a compact space with $Z$ totally disconnected. Then, there exists only one continuous surjective map of the form $id+\phi: X+_{f}Ends(X) \rightarrow X+_{g}Z$.
\end{cor}

\begin{proof}Just apply the last proposition to the subspace $X \cup g(X)$.
\end{proof}

So $X+_{f}Ends(X)$ is the Freudenthal compactification of $X$.

\begin{prop}\label{extensaoends}Let $X_{1},X_{2}$ be locally compact, connected and locally connected spaces and $j: X_{1} \rightarrow X_{2}$ be a proper continuous map. Then, there exists a unique continuous extension to the Freudenthal compactifications: $X_{1}+_{f_{1}}Ends(X_{1}) \rightarrow X_{2}+_{f_{2}} Ends(X_{2})$.
\end{prop}

\begin{proof}Since $j$ is proper, $\forall K \subseteq X_{2}$ compact, $j^{-1}(K)$ is also compact. If $U \in \pi_{0}^{u}(X_{1} - j^{-1}(K))$, then $j(U)$ is connected, which implies that it is contained in a connected component of $X_{2}-K$. If $j(U)$ is bounded, then $Cl_{X_{2}}(j(U))$ is compact, which implies that $j^{-1}(Cl_{X_{2}}(j(U)))$ is compact (since $j$ is proper). But $j^{-1}(Cl_{X_{2}}(j(U))) \supseteq U$, contradicting the fact that $U \in \pi_{0}^{u}(X_{1}-j^{-1}(K))$. So $j(U)$ is contained in an (unique) element of $\pi_{0}^{u}(X_{2}-K)$. Consider the map $j_{K}: \pi_{0}^{u}(X_{1}-j^{-1}(K)) \rightarrow \pi_{0}^{u}(X_{2} - K)$ defined by $j_{K}(U)$ equal to the connected component of $j(U)$. Since $\pi_{0}^{u}(X_{1} - j^{-1}(K))$ is discrete, it follows that $j_{K}$ is continuous.

Let $F$ be a closed subset of $X_{2}$. Then, we have that $f_{j^{-1}(K)} \circ j^{-1}(F) = \{U \in \pi_{0}^{u}(X_{1}-j^{-1}(K)): U \cap j^{-1}(F)$ is unbounded$\}$. But $U \cap j^{-1}(F)$ unbounded implies that $j(U \cap j^{-1}(F)) \subseteq j(U) \cap F \subseteq j_{K}(U) \cap F$ is unbounded. Hence, $U \in f_{j^{-1}(K)} \circ j^{-1}(F)$ implies $j_{K}(U) \in f_{K}(F)$ and then $U \in j_{K}^{-1} \circ f_{K}(F)$. So $f_{j^{-1}(K)} \circ j^{-1}(F) \subseteq j_{K}^{-1} \circ f_{K}(F)$. In another words, we have the diagram:

$$ \xymatrix{ Closed(X_{1}) \ar[r]^<<{ \ \ \ \ f_{j^{-1}(K)}} & Closed(\pi_{0}^{u}(X_{1}-j^{-1}(K))) \\
            Closed(X_{2}) \ar[r]^<<{ \ \ \ \ \ \ f_{K}} \ar[u]^{j} \ar@{}[ur]|{\subseteq} & Closed(\pi_{0}^{u}(X_{2}-K)) \ar[u]^{j_{K}^{-1}} } $$

Then, the map $j+j_{K}: X_{1}+_{f_{j^{-1}(K)}} \pi_{0}^{u}(X_{1}-j^{-1}(K)) \rightarrow X_{2}+_{f_{K}} \pi_{0}^{u}(X_{2}-K)$ is continuous. It is clear that, $\forall K_{1} \subseteq K_{2} \subseteq X_{2}$ compact subspaces, the diagram commutes:

$$ \xymatrix{ X_{1} \!\! +_{f_{j^{-1}(K_{2})}} \!\! \pi_{0}^{u}(X_{1} \! - \! j^{-1}(K_{2})) \ar[r]_*++{\labelstyle \ id+ \psi_{j^{-1}(K_{1})j^{-1}(K_{2})} } \ar[d]^{j+j_{K_{2}}} & X_{1} \!\! +_{f_{j^{-1}(K_{1})}} \!\! \pi_{0}^{u}(X_{1} \! - \! j^{-1}(K_{1})) \ar[d]^{j+j_{K_{1}}}  \\
           X_{2} \!\! +_{f_{K_{2}}} \!\! \pi_{0}^{u}(X_{2} \! - \! K_{2}) \ar[r]^{id+ \psi_{K_{1}K_{2}}}  & X_{2} \!\! +_{f_{K_{1}}} \!\! \pi_{0}^{u}(X_{2} \! - \! K_{1}) } $$

So it induces a continuous map $\tilde{j}: X_{1}+_{f_{1}}Ends(X_{1}) \rightarrow X_{2}+_{f_{2}}Ends(X_{2})$ that extends $j$. The uniqueness comes from the fact that the space is dense on its compactification.
\end{proof}

The existence of limits described on the last section gives us an easier way to construct the Freudenthal compactification that works for every locally compact Hausdorff space:

\begin{prop}Let $X$ be a locally compact Hausdorff space. Then, there exists a universal compactification of $X$ such that the remainder is totally disconnected.
\end{prop}

\begin{proof}Let $\{X+_{f_{i}}Y_{i}\}_{i \in \Gamma}$ be the collection of all compactifications of $X$ such that $Y_{i}$ is totally disconnected. Since all of them are quotients of the Stone-$\check{C}$ech compactification, it follows that this collection is actually a set. So there exists a product $X+_{f}Y$ for those spaces in the category $Sum_{0}(X)$. Since this product is the closure of $X$ in the pullback in the category $Top$, we have that $X+_{f}Y$ is compact and $Y$ is a subspace of $\prod Y_{i}$, which implies that $Y$ is totally disconnected. The projection maps are the unique maps to the spaces $X+_{f_{i}}Y_{i}$ that are the identity on $X$, since $X$ is dense. Thus, $X+_{f}Y$ is the universal compactification of $X$ such that the remainder is totally disconnected.
\end{proof}

\subsection{General case functors}

\begin{defi}Let $X$ be a locally compact Hausdorff space. Let $Comp(X)$ be the category whose objects are compact spaces of the form $X+_{f}W$, where $W$ is a Hausdorff compact space and morphisms are continuous maps that are the identity on $X$. Let $T_{2}Comp(X)$ be the full subcategory of $Comp(X)$ whose objects are Hausdorff spaces.
\end{defi}

\begin{prop}\label{pullbackdecompactificacoes}Let $X$ and $Y$ be two locally compact Hausdorff spaces and $\pi: X \rightarrow Y$ a proper map. Then, the map $\Pi: Comp(Y) \rightarrow Comp(X)$ such that $\Pi(Y+_{f}Z) = X+_{f^{\ast}}Z$ and, for $id+\varpi: Y+_{f}Z \rightarrow Y+_{g}W$, $\Pi(id+\varpi) = id+\varpi: X+_{f^{\ast}}Z \rightarrow X+_{g^{\ast}}W$ is a functor.
\end{prop}

\begin{proof}Let $Y+_{f}Z \in Comp(Y)$. Since $\pi$ is a proper map and $Y+_{f}Z$ is compact, it follows by \textbf{Proposition \ref{compact}} that $X+_{f^{\ast}}Z$ is compact.

Let $id+\varpi: Y+_{f}Z \rightarrow Y+_{g}W$ be a continuous map. The following diagrams commute:

$$ \xymatrix{ X \ar[r]^{id} \ar[d]^{\pi} & X \ar[d]_{\pi}  & & Z \ar[r]^{\varpi} \ar[d]^{id} & W \ar[d]_{id} \\
           Y \ar[r]^{id} & Y & & Z \ar[r]^{\varpi} & W} $$

By the Cube Lemma we have that the map $id+\varpi: X+_{f^{\ast}}Z \rightarrow X+_{g^{\ast}}W$ is continuous. Thus, it follows that $\Pi$ is a functor.
\end{proof}

\begin{prop}If $\pi$ is continuous, then the functor $\Pi$ sends Hausdorff spaces to Hausdorff spaces.
\end{prop}

\begin{proof}Immediate from \textbf{Proposition \ref{hausdorffpullback}}.
\end{proof}

\begin{prop}If $Y$ is dense in $Y+_{f}Z$, then $X$ is dense in $X+_{f^{\ast}}Z$. In particular, if $\pi$ is continuous, then the functor $\Pi$ sends compactifications of $X$ to compactifications of $X$.
\end{prop}

\begin{proof}Immediate from the definition of pushforward.
\end{proof}

We will use $\Pi$ as any of its restrictions. It may not cause confusion.

\begin{prop}If $Y+_{f}W$ is metrizable, $X$ has $countable$ basis and $X+_{f^{\ast}}W$ is Hausdorff, then $X+_{f^{\ast}}W$ is metrizable.
\end{prop}

\begin{proof}Immediate from \textbf{Proposition \ref{uniaometrizavel}}.
\end{proof}

\subsection{Coarse spaces}

Here we see cases when we put enough structure on the topological spaces such that the functor $\Pi$ defined at the section above is restricted to an isomorphism of categories.

\subsubsection{Subsets that are coarsely close}

\begin{defi}Let $(X,\varepsilon)$ be a coarse space and $A,B \subseteq X$. We say that $A \preceq B$ if there exists $e \in \varepsilon$ such that $A \subseteq \B(B,e)$.
\end{defi}

\begin{prop}$\preceq$ is a preorder on the power set of $X$.
\end{prop}

\begin{proof}Let $A \subseteq X$. We have that $\Delta X \in \varepsilon$. So $A = \B(A,\Delta X)$, which implies that $A \preceq A$.

Let $A,B,C \subseteq X$ such that $A \preceq B$ and $B \preceq C$. There exists $e,e' \in \varepsilon$ such that $A \in \B(B,e)$ and $B \in \B(C,e')$. Let $a \in A$. There exists $b \in B$ such that $(a,b) \in e$ and there exists $c \in C$ such that $(b,c) \in E'$. So $(a,c) \in e' \circ e$, which implies that $A \in \B(C, e'\circ e)$ and then $A \preceq C$.

Thus $\preceq$ is a preorder.
\end{proof}

We denote by $\sim$ the equivalence relation defined by $A \sim B$ if $A \preceq B$ and $B\preceq A$.

\begin{obs}If $X$ is a metric space and $\varepsilon$ is the bounded coarse structure associated with the metric, then this equivalence relation is the same as the finite Hausdorff distance.
\end{obs}

\begin{lema}If $A \subseteq B \subseteq X$, then $A \preceq B$.
\end{lema}

\begin{proof}$A \subseteq B = \B(B,\Delta X)$.
\end{proof}

\begin{lema}Let $(X,\varepsilon)$ be a proper coarse Hausdorff space. If $A \subseteq X$, then $A \sim Cl_{X}(A)$.
\end{lema}

\begin{proof}Since $A \subseteq Cl_{X}(A)$, it follows that $A \preceq Cl_{X}(A)$. Let $u \in \varepsilon$ be a neighborhood of $\Delta X$ and let $a \in Cl_{X}(A)$. We have that $(a,a) \in \Delta Cl_{X}(A) = Cl_{X\times X} (\Delta A)$. Since $u \in \varepsilon$ is a neighborhood of $(a,a)$, there exists $V$ an open neighborhood of a such that $V \times V \subseteq u$. Since $(a,a) \in Cl_{X\times X} (\Delta A)$, there exists $a' \in A$ such that $(a',a') \in V \times V$. So $(a,a') \in V \times V$, which implies that $a \in \B(A,u)$. Then $Cl_{X}(A) \preceq A$.

Thus $A \sim Cl_{X}(A)$.
\end{proof}

\begin{lema}Let $(X,\varepsilon)$, and $(Y,\zeta)$ be coarse spaces, $\pi: Y \rightarrow X$ a coarse map and $A,B \subseteq Y$. If $A \preceq B$ in $\zeta$, then $\pi(A) \preceq \pi(B)$ in $\varepsilon$.
\end{lema}

\begin{proof}If $A \preceq B$, then there exists $e \in \zeta$ such that $A \subseteq \B(B,e)$, which implies that $\pi(A) \subseteq \B(\pi(B),\pi(e))$. Since $\pi$ is a coarse map, it follows that $\pi(e) \in \varepsilon$. Thus $\pi(A) \preceq \pi(B)$ in $\varepsilon$.
\end{proof}

\begin{cor}\label{preservarelacao}Let $(X,\varepsilon)$ and $(Y,\zeta)$ be proper coarse spaces and a coarse map $\pi: Y \rightarrow X$. If $A \subseteq X$, then $\pi(A) \sim \pi(Cl_{Y}(A))$ in $\varepsilon$. \eod
\end{cor}

\subsubsection{Subsets of compactifications that are coarsely close}

\begin{lema}\label{clusterponitsontheboundary}Let $X$ be a locally compact paracompact space, $X+_{f}W$ a Hausdorff compactification of $X$ and $e \subseteq X \times X$ proper. If $\{(x_{\gamma},y_{\gamma})\}_{\gamma \in \Gamma}$ is a net contained in $e$ such that $\{x_{\gamma}\}_{\gamma \in \Gamma}$ has cluster points only in $W$, then $\{y_{\gamma}\}_{\gamma \in \Gamma}$ has cluster points only in $W$.
\end{lema}

\begin{obs}Since $e$ is proper if and only if $e^{-1}$ is proper, it follows that if $\{(x_{\gamma},y_{\gamma})\}_{\gamma \in \Gamma}$ is a net contained in $e$ such that $\{y_{\gamma}\}_{\gamma \in \Gamma}$ has cluster points only in $W$, then $\{x_{\gamma}\}_{\gamma \in \Gamma}$ has cluster points only in $W$.
\end{obs}

\begin{proof}Let $x \in X$ be a cluster point of $\{y_{\gamma}\}_{\gamma \in \Gamma}$  and let $\{y_{\gamma}\}_{\gamma\in \Gamma'}$ be a subnet of $\{y_{\gamma}\}_{\gamma \in \Gamma}$ that converges to $x$. We have that the set $A = \{y_{\gamma}\}_{\gamma\in \Gamma'}$ is topologically bounded in $X$, which implies that $\B(A,e)$ is also topologically bounded in $X$. However, $\forall \gamma \in \Gamma', x_{\gamma} \in \B(A,e)$, which implies that the set of cluster points of the net $\{x_{\gamma}\}_{\gamma \in \Gamma'}$ must be in $Cl_{X+_{f}W}(\B(A,e)) \subseteq X$. Since $\{x_{\gamma}\}_{\gamma \in \Gamma'}$ is a subnet of $\{x_{\gamma}\}_{\gamma \in \Gamma}$, it contradicts the hypothesis.
\end{proof}

Let $X$ be a locally compact paracompact Hausdorff space and let $W = X+_{f}Y$ be a Hausdorff compactification of $X$. We denote by $\varepsilon_{f}$ the coarse structure on $X$ induced by $W$ (instead of $\varepsilon_{W}$).

\begin{lema}Let $X$ be a locally compact paracompact space, $X+_{f}W$ a Hausdorff compactification of $X$, $A \subseteq X$ and $e \in \varepsilon_{f}$ surjective on the first coordinate. Then $Cl_{X+_{f}W}(A) - X = Cl_{X+_{f}W}(\B(A,e)) - X$.
\end{lema}

\begin{proof}Let $x \in Cl_{X+_{f}W}(A) - X$. There exists a net $\{x_{\gamma}\}_{\gamma \in \Gamma} \subseteq A$ that converges to $x$. For every $\gamma \in \Gamma$, choose $y_{\gamma} \in \B(A,e)$ such that $(y_{\gamma},x_{\gamma}) \in e$ (it is possible since $e$ is surjective on the first coordinate). Since $e$ is perspective, we have that the net $\{y_{\gamma}\}_{\gamma\in \Gamma} \subseteq \B(A,e)$ converges to $x$. So $x \in  Cl_{X+_{f}W}(\B(A,e)) - X$.

Let $x \in Cl_{X+_{f}W}(\B(A,e)) - X$.  There exists a net $\{x_{\gamma}\}_{\gamma \in \Gamma} \subseteq \B(A,e)$ that converges to $x$. For every $\gamma \in \Gamma$, choose $y_{\gamma} \in A$ such that $(x_{\gamma},y_{\gamma}) \in e$. Since $e$ is perspective, we have that the net $\{y_{\gamma}\}_{\gamma\in \Gamma} \subseteq A$ converges to $x$. So $x \in  Cl_{X+_{f}W}(A) - X$.

Thus $Cl_{X+_{f}W}(A) - X = Cl_{X+_{f}W}(\B(A,e)) - X$.
\end{proof}

\begin{lema}Let $X$ be a locally compact paracompact space, $X+_{f}W$ a Hausdorff compactification of $X$ and $A,B \subseteq X$. If $A \preceq B$ in $\varepsilon_{f}$, then $Cl_{X+_{f}W}(A) - X \subseteq Cl_{X+_{f}W}(B) - X$.
\end{lema}

\begin{proof}Since $A \preceq B$, there exists $e \in \varepsilon_{f}$ such that $A \subseteq \B(B,e) \subseteq \B(B,e')$, with $e' = e \cup \Delta X$ (which is surjective on the first coordinate). So we have that $Cl_{X+_{f}W}(A) - X \subseteq Cl_{X+_{f}W}(\B(B,e')) - X = Cl_{X+_{f}W}(B) - X$.
\end{proof}

\begin{defi}Let $X$ be a locally compact paracompact Hausdorff space, $\varepsilon$ be a coarsely connected proper coarse structure of $X$ and $X+_{f}W$ a Hausdorff compactification of $X$ . We say that $X+_{f}W$ is perspective if $\varepsilon_{f} \supseteq \varepsilon$. Let $Pers(\varepsilon)$ be the full subcategory of $Comp(X)$ whose objects are perspective compactifications of $(X,\varepsilon)$.
\end{defi}

\begin{obs}Such compactifications are called coarse in \cite{Ro}.
\end{obs}

\begin{prop}\label{closesameboundary} Let $X+_{f}W \in Pers(\varepsilon)$ and $A,B \subseteq X$. If $A \preceq B$ in $\varepsilon$, then $Cl_{X+_{f}W}(A) - X \subseteq Cl_{X+_{f}W}(B) - X$. \eod
\end{prop}

\begin{obs}In other words, if  $A,B \in Closed(X)$ and $A \preceq B$ in $\varepsilon$, then $f(A) \subseteq f(B)$. In particular, if $A \sim B$, then $f(A) = f(B)$.
\end{obs}

\subsubsection{Perspective compactifications}

Let's consider $(X,\varepsilon)$ and $(Y,\zeta)$ coarsely connected proper coarse spaces and $X+_{f}W\in Pers(\varepsilon)$.

\begin{prop}Let $\pi,\pi': Y \rightarrow X$ be coarse maps and $f^{\ast}$ and $f^{'\ast}$ their respective pullbacks. If $\pi$ is close to $\pi'$, then $f^{\ast} = f^{'\ast}$
\end{prop}

\begin{proof}Let $F \in Closed(X)$. Since $\pi$ is close to $\pi'$, we have that $\pi(F)\sim \pi'(F)$ in $\varepsilon_{f}$. We have also that $\pi(F)\sim Cl_{X}(\pi(F))$ and $\pi'(F)\sim Cl_{X}(\pi'(F))$, which implies that $Cl_{X}(\pi(F)) \sim Cl_{X}(\pi'(F))$. So $f^{\ast}(F) = f(Cl_{X}(\pi(F))) = f(Cl_{X}(\pi'(F))) = f'^{\ast}(F) $. Thus $f^{\ast} = f^{'\ast}$.
\end{proof}

So let's consider the coarse equivalence $\pi: Y \rightarrow X$ and $f^{\ast}$ its pullback.

\begin{prop}$Y+_{f^{\ast}}W$ is Hausdorff.
\end{prop}

\begin{proof}Let $K \subseteq Y$ be a compact space. Since $\pi$ is a coarse equivalence, $\pi(K)$ is bounded in $X$, which implies that $Cl_{X}(\pi(K))$ is compact. Since $X+_{f}W$ is Hausdorff, we have that $f^{\ast}(K) = f(Cl_{X}(\pi(K))) = \emptyset$.

Let $a,b \in W$ with $a \neq b$. Since $X+_{f}W$ is Hausdorff, there exists $A,B \in Closed(X)$ such that $A\cup B = X$, $b \notin f(A)$ and $a \notin f(B)$. Let $A' = Cl_{Y}(\pi^{-1}(A))$ and $B' = Cl_{Y}(\pi^{-1}(B))$. We have that $A' \cup B' = X$. By \textbf{Corollary \ref{preservarelacao}}, we have that $A \supseteq \pi(\pi^{-1}(A)) \sim \pi(Cl_{Y}(\pi^{-1}(A)))$, which implies that $A \succeq Cl_{X}(\pi(Cl_{Y}(\pi^{-1}(A))))$. Then $f(A) \supseteq f(Cl_{X}(\pi(Cl_{Y}(\pi^{-1}(A))))) = f^{\ast}(Cl_{Y}(\pi^{-1}(A))) = f^{\ast}(A')$, which implies that $b \notin f^{\ast}(A')$. Analogously, $a \notin f^{\ast}(B')$.

Thus $Y+_{f^{\ast}}W$ is Hausdorff.
\end{proof}

So the functor $\Pi$ has a restriction $\Pi: Pers(\varepsilon) \rightarrow T_{2}Comp(Y)$ (compactness is preserved since $\pi$ is a proper map).

Let $\varpi: X \rightarrow Y$ be a quasi-inverse of $\pi$.

\begin{prop}$(f^{\ast})^{\ast} = f$, where the second pullback is relative to $\varpi$.
\end{prop}

\begin{proof}Let $A \in Closed(X)$. Then $(f^{\ast})^{\ast}(A) = f(Cl_{X}(\pi(Cl_{Y}(\varpi(A)))))$. But $Cl_{X}(\pi(Cl_{Y}(\varpi(A))))  \sim Cl_{X}(\pi(\varpi(A)))$. By \textbf{Proposition \ref{closesameboundary}}, we have that $f(Cl_{X}(\pi(Cl_{Y}(\varpi(A))))) = f(Cl_{X}(\pi(\varpi(A)))) = f^{\ast\ast}(A)$, where $f^{\ast\ast}$ is the pullback with respect to $\pi\circ \varpi$ and $id_{W}$. Since $\pi$ and $\varpi$ are quasi-inverses, we have that $f^{\ast\ast}(A) = f(A)$. Thus $(f^{\ast})^{\ast} = f$.
\end{proof}

\begin{prop}If $\pi$ is continuous, then $\zeta_{f^{\ast}} \supseteq \zeta$.
\end{prop}

\begin{proof}Let $e \in \zeta$. By the \textbf{Proposition \ref{propercontrolled}} $e$ is proper. Let $\{(x_{\gamma},y_{\gamma})\}_{\gamma \in \Gamma}  \subseteq e$ be a net such that $\{x_{\gamma}\}_{\gamma \in \Gamma}$ converges to $x \in W$. Let $y$ be a cluster point of $\{y_{\gamma}\}_{\gamma \in \Gamma}$. Since $e$ is proper, it follows that $y \in W$. Since $\pi$ is continuous, we have that $\pi+id: Y+_{f^{\ast}}W \rightarrow X+_{f}W$ is continuous. So $\{(\pi(x_{\gamma}),\pi(y_{\gamma}))\}_{\gamma \in \Gamma}  \subseteq \pi(e) \in \varepsilon$, the net $\{\pi(x_{\gamma})\}_{\gamma \in \Gamma}$ converges to $x$ and $y$ is a cluster point of  $\{\pi(y_{\gamma})\}_{\gamma \in \Gamma}$. Since $\pi(e) \in \varepsilon$, $X+_{f}Y \in Pers(\varepsilon)$ and  $\{\pi(x_{\gamma})\}_{\gamma \in \Gamma}$ converges to $x$, it follows that $y = x$. So $x$ is the unique cluster point of $\{y_{\gamma}\}_{\gamma \in \Gamma}$, which implies that this net converges to $x$.

Thus $e \in \zeta_{f^{\ast}}$ and then $\zeta_{f^{\ast}} \supseteq \zeta$.
\end{proof}

Let $\Lambda: Comp(Y) \rightarrow Comp(X)$ be the functor given by the pullback relative to $\varpi$.

\begin{teo}\label{Teorema1}If $\pi$ and $\varpi$ are continuous, then the categories $Pers(\varepsilon)$ and $Pers(\zeta)$ are isomorphic.
\end{teo}

\begin{proof}We have that $\Pi$ and $\Lambda$ preserve the perspectivity property, $\forall X+_{f}W \in Pers(\varepsilon)$, $\Lambda\circ \Pi(X+_{f}W) = X+_{(f^{\ast})^{\ast}}W = X+_{f}W$ and $\forall Y+_{g}W \in Pers(\zeta)$, $\Pi \circ \Lambda(Y+_{g}W) = Y+_{(g^{\ast})^{\ast}}W = Y+_{g}W$. It is clear also that every morphism is preserved by $\Pi \circ \Lambda$ and $\Lambda\circ \Pi$. Thus, $\Pi$ and $\Lambda$ are inverses, which implies that $Pers(\varepsilon)$ and $Pers(\zeta)$ are isomorphic.
\end{proof}

\begin{obs}Our construction is the same as the one used on Theorem 7.1 of \cite{GM} for the case where the spaces $X$ and $Y$ are uniformly contractible ANR metric spaces. On the language that we are using,  Theorem 7.1 of \cite{GM} says that the functors $\Pi$ and $\Lambda$ sends controlled $\mathcal{Z}$-compactifications to controlled $\mathcal{Z}$-compactifications (controlled in their sense is equivalent to be, in our sense, perspective with respect to the bounded coarse structure).
\end{obs}

The hypotheses of the theorem are enough for working with discrete spaces, being useful to study discrete groups. Another class of spaces that it is enough is uniformly contractible ANR spaces with proper metric and finite macroscopic dimension:

\begin{prop}(Corollary 5.4 of \cite{GM}) Let $X$ and $Y$ be uniformly contractible ANR proper metric spaces with finite macroscopic dimension and $\varepsilon$ and $\zeta$ are the bounded coarse structures associated to the metrics of $X$ and $Y$, respectively. If $\pi: X \rightarrow Y$ is a coarse equivalence with coarse inverse $\varpi$, then there exists a continuous coarse equivalence $\pi': X \rightarrow Y$ and a continuous coarse inverse $\varpi'$ such that $\pi'$ is close to $\pi$ and $\varpi'$ is close to $\varpi$.
\end{prop}

\begin{cor}Let $X$ and $Y$ be uniformly contractible ANR metric  spaces with finite macroscopic dimension that are coarse equivalent and $\varepsilon$ and $\zeta$ are the bounded coarse structures associated to the metrics of $X$ and $Y$, respectively. Then the categories $Pers(\varepsilon)$ and $Pers(\zeta)$ are isomorphic.
\end{cor}

\begin{proof}If $\pi: X \rightarrow Y$ is a coarse equivalence and $\varpi$ its quasi-inverse, then there exists $\pi'$ close to $\pi$ and $\varpi'$ close to $\varpi$ that are continuous. So $Pers(\varepsilon)$ and $Pers(\zeta)$ are isomorphic by the previous theorem.
\end{proof}

\begin{obs}Since $\pi$ is close to $\pi'$ and $\varpi$ is close to $\varpi'$, the functors given by the pullbacks of $\pi$ and $\varpi$ are the same of the functors given by the pullbacks of $\pi'$ and $\varpi'$, respectively. So they are isomorphisms.
\end{obs}

Let's consider the case where $\pi$ or $\varpi$ may not be continuous.

\begin{defi}Let $X$ be a space with countable basis. Let $MComp(X)$, and $MPers(\varepsilon)$ be the full subcategories of $Comp(X)$ and $Pers(\varepsilon)$, respectively, such that the objects are metrizable spaces.
\end{defi}

Let's consider $X$ and $Y$with countable basis and $X+_{f}W \in MPers(\varepsilon)$.

\begin{prop}$\zeta_{f^{\ast}} \supseteq \zeta$.
\end{prop}

\begin{proof}Let $e \in \zeta$. Let $(x,y) \in Cl_{(Y+_{f^{\ast}}W)^{2}}(e) \cap W^{2}$. There exists a sequence $\{(x_{n},y_{n})\}_{n\in \N}\subseteq e$ converging to $(x,y)$. We have that $\{x\} = f^{\ast}(\{x_{n}\}_{n \in \N}) = f(Cl_{X}(\{\pi(x_{n})\}_{n\in\N}))$ and $\{y\}= f^{\ast}(\{y_{n}\}_{n \in \N}) = f(Cl_{X}(\{\pi(y_{n})\}_{n\in\N}))$. So $x$ and $y$ are respectively the unique cluster points of $\{\pi(x_{n})\}_{n \in \N}$ and $\{\pi(y_{n})\}_{n \in \N}$ that are in $W$. But all subsequences of $\{\pi(x_{n})\}_{n \in \N}$ and $\{\pi(y_{n})\}_{n \in \N}$ are not bounded (since $\pi$ is a coarse equivalence), which implies that $x$ and $y$ are respectively the unique cluster points of $\{\pi(x_{n})\}_{n \in \N}$ and $\{\pi(y_{n})\}_{n \in \N}$. Then $\{(\pi(x_{n}),\pi(y_{n}))\}_{n \in \N}$ converges to $(x,y)$. But $\{(\pi(x_{n}),\pi(y_{n}))\}_{n \in \N} \subseteq \pi(e) \in \varepsilon$, which implies that $x = y$. Thus $e \in \zeta_{f^{\ast}}$.
\end{proof}

\begin{teo}\label{Teorema2}If $X$ and $Y$ have countable basis, then $MPers(\varepsilon)$ and $MPers(\zeta)$ are isomorphic. \eod
\end{teo}

We end this section with a technical lemma that is useful on \textbf{Section \ref{coproducts}}.

\begin{lema}\label{pushforwradpersp}Let $Y+_{f}W \in Pers(\zeta)$. If $\pi$ and $\varpi$ are continuous and $\pi$ is surjective, then $\Lambda(Y+_{f}W) = X+_{f_{\ast}}W$, where the pushforward is with respect to $\pi$ and $id_{W}$. 
\end{lema}

\begin{proof}Since $\pi$ is continuous, the maps $\pi\!+\!id_{W}\!: Y\!+_{f}W \!\rightarrow X\!+_{f_{\ast}}\!W$ and $\pi\!+\!id_{W}\!:  \Pi (\Lambda(Y\!+_{f}W)) \!\rightarrow \Lambda(Y\!+_{f}W)$ are continuous (by the properties of the pushforward and the pullback, respectively) and, since $\pi$ and $\varpi$ are continuous maps and coarse inverses, $\Pi(\Lambda(Y+_{f}W)) = Y+_{f}W$. Since $\Lambda(Y+_{f}W) = X+_{f^{\ast}}W$, where the pullback is with respect to $\varpi$ and $id_{W}$, and by \textbf{Proposition \ref{propriedadeuniversalpushforward}}, the diagram commutes:

$$ \xymatrix{ Y+_{f}W \ar[r]^{\pi+id} \ar@{}[d]|{\parallel} & X+_{f_{\ast}}W \ar[d]^{id+id}  \\
           \Pi(\Lambda(Y\!+_{f}W)) \ar[r]^{ \ \ \ \pi+id} & X+_{f^{\ast}}W } $$
           
Since $\pi$ is surjective, $X+_{f_{\ast}}W$ is compact (\textbf{Proposition \ref{pushforwardcompact}}). Since $X+_{f^{\ast}}W$ is Hausdorff and $X+_{f_{\ast}}W$ is compact, we have that the map $id_{X}+id_{W}: X+_{f_{\ast}}W \rightarrow X+_{f^{\ast}}W$ is a homeomorphism.
\end{proof}

\subsection{Karlsson property}

\begin{defi}Let $(X,d)$ be a metric space. A coarse arc between $x,y \in X$ is a coarse embedding $\lambda: [0,t] \rightarrow X$, where $[0,t]$ and $X$ have the coarse structures given by their metrics, such that $\lambda(0) = x$ and $\lambda(t) = y$. A set $\Upsilon$ of coarse arcs on $X$ is equicoarse if both conditions happens:

\begin{enumerate}
    \item $\forall r > 0$, $\exists s > 0$ such that $\forall \lambda \in \Upsilon$, $\forall a,b \in Dom \ \lambda$, if  $|a-b| < r$, then $d(\lambda(a),\lambda(b)) < s$.
    \item $\forall r > 0$, $\exists s > 0$ such that $\forall \lambda \in \Upsilon$, $\forall a,b \in Dom \ \lambda$, if  $d(\lambda(a),\lambda(b)) < r$, then $|a-b| < s$.
\end{enumerate}

\end{defi}

\begin{defi}Let $(X,d)$ be a proper metric space and $Y \subseteq X$. Let $\Upsilon$ be a set of equicoarse arcs on $X$ such that $\forall x,y \in Y$, $\exists \lambda\in \Upsilon$ a coarse arc between $x$ and $y$. A compactification $X+_{f}W$ has the coarse Karlsson property with respect to $\Upsilon$ if $\forall u \in \U_{f}$, there exists a bounded set $S \subseteq X$ such that every coarse arc in $\Upsilon$ that do not intersect $S$ is $u$-small, where $\U_{f}$ is the uniform structure compatible with the topology of $X+_{f}W$. We say that $X+_{f}W$ has the coarse Karlsson property if it has the coarse Karlsson property with respect to some $\Upsilon$ such that $\forall x,y \in X$, there is a coarse arc in $\Upsilon$ between $x$ and $y$.
\end{defi}

\begin{obs}If the space $X$ is geodesic and $\Upsilon$ is the set of geodesics on $X$, then the coarse Karlsson property with respect to $\Upsilon$ is the original Karlsson property. It appears on $\cite{Ka}$ on Floyd compactifications.
\end{obs}

\begin{prop}Let $(X,d)$ be a proper metric space and a compactification $X+_{f}W$ with the coarse Karlsson property. Then, $X+_{f}W \in Pers(\varepsilon_{d})$.
\end{prop}

\begin{proof}Let $\Upsilon$ be a set of coarse arcs such that $\forall x,y \in X$, there is a coarse arc in $\Upsilon$ between $x$ and $y$ and $X+_{f}W$ has the coarse Karlsson property with respect to $\Upsilon$. Let $e \in \varepsilon_{d}$. By the \textbf{Proposition \ref{propercontrolled}} $e$ is proper. Let $\{(x_{\gamma},y_{\gamma})\}_{\gamma \in \Gamma} \subseteq e$ be a net such that $\{x_{\gamma}\}_{\gamma \in \Gamma}$ converges to a point $x \in W$. Let $y$ be a cluster point of $\{y_{\gamma}\}_{\gamma \in \Gamma}$. Suppose that $y \neq x$. Then there exists $u \in \U_{f}$ such that $(x,y) \notin u$. Let $v \in \U_{f}$ such that  $v^{3}\subseteq u$.

By the Karlsson property, there is a bounded set $S \subseteq X$ such that every coarse arc in $\Upsilon$ that does not intersect $S$ is $v$-small. Let $\gamma_{0} \in \Gamma$ such that $\forall \gamma > \gamma_{0}$, $(x,x_{\gamma}) \in v$ and take $\gamma_{1} > \gamma_{0}$ such that $(y_{\gamma_{1}},y) \in v$. If $(x_{\gamma_{1}},y_{\gamma_{1}}) \in v$, then $(x,y) \in v^{3} \subseteq u$, a contradiction. So $(x_{\gamma_{1}},y_{\gamma_{1}}) \notin v$. If $\lambda \in \Upsilon$ is a coarse arc between $x_{\gamma_{1}}$ and $y_{\gamma_{1}}$, then $\lambda$ must intersect $S$. Since $\{x_{\gamma}\}_{\gamma \in \Gamma}$ have no cluster points in $X$ and $X$ is proper, for every $\alpha > 0$, for every $\gamma_{0} \in \Gamma$, we can take $\gamma_{1} > \gamma_{0}$ such that $d(x_{\gamma_{1}},S) > \alpha$, which implies that $diam(Im \ \lambda) > \alpha$. Then, $I = \{diam(Im \ \lambda): \lambda$ is an coarse arc of $\Upsilon$ between $x_{\gamma}$ and $y _{\gamma}$, for some $\gamma \in \Gamma\}$ is unbounded.

On the other hand, let $k = \sup \{d(x_{\gamma},y_{\gamma}): \gamma \in \Gamma\}$. Since the maps in $\Upsilon$ are equicoarse, for $k > 0$, there exists $t > 0$ such that $\forall \lambda  \in \Upsilon$, if $a \geqslant 0$ such that $\lambda(0) = x_{\gamma}$ and $\lambda(a) = y_{\gamma}$, for some $\gamma \in \Gamma$, then $a < t$. So the domain of every coarse arc of $\Upsilon$ between $x_{\gamma}$ and $y_{\gamma}$, for some $\gamma \in \Gamma$, is contained in $[0,t]$. We have also that, for $t > 0$, for every $\lambda \in \Upsilon$, there exists $s > 0$, such that if $a,b \in [0,t]$, then $d(\lambda(a),\lambda(b)) < s$, a contradiction to the fact that $I$ is unbounded.

So $y = x$, which implies that $\{y_{\gamma}\}_{\gamma \in \Gamma}$ converges to $x$ and then $e$ is perspective. Thus $X+_{f}W \in Pers(\varepsilon_{d})$.
\end{proof}

\begin{obs}For the Karlsson property, this result is known for equivariant compactifications of groups, since it implies that the action of the group on the boundary has the convergence property (Proposition 4.3.1 of \cite{GP}) and equivariant compactifications with the convergence property have the perspective property (Proposition 7.5.4 of \cite{Ge2}).
\end{obs}

Let $(X,d)$ be a proper metric space and $Y$ a $r$-quasi-dense subset of $X$. Let $\Upsilon$ be a set of equicoarse arcs in $X$ such that $\forall x,y \in Y$, $\exists \lambda\in \Upsilon$ a coarse arc between $x$ and $y$. If $\lambda: [0,t] \rightarrow X$ and $x,y \in X$, we define $\lambda_{x,y}: [0,t+2r] \rightarrow X$ by $\lambda_{x,y}(0) = x$, $\lambda_{x,y}((0,r]) = \lambda(0)$, $\lambda_{x,y}(a) = \lambda(a-r)$, for $a \in [r,t+r]$, $\lambda_{x,y}([t+r,t+2r)) = \lambda(t)$ and $\lambda_{x,y}(t+2r) = y$. Let $\Upsilon_{X} = \{\lambda_{x,y}: \lambda \in \Upsilon, \ d(x,\lambda(0)) < r, \ d(y,\lambda(t)) < r, \ t = \max(dom \ \lambda)\}$.

\begin{lema}$\Upsilon_{X}$ is a set of equicoarse arcs.
\end{lema}

\begin{proof}Let $k > 0$. There exists $s > 0$ such that $\forall \lambda \in \Upsilon$,  $d(a,b) < k$ implies $d(\lambda(a),\lambda(b)) < s$. Then, $\forall \lambda_{x,y} \in \Upsilon_{X}$,  $d(a,b) < k$ implies $d(\lambda_{x,y}(a),\lambda_{x,y}(b)) < s+2r$.

Let $k > 0$. There exists $s > 0$ such that $\forall \lambda \in \Upsilon$, $d(\lambda(a),\lambda(b)) < k+2r$ implies $d(a,b) < s$. Let $\lambda: [0,t] \rightarrow X$, $\lambda \in \Upsilon$ and $x,y\in X$ such that $d(x,\lambda(0)) < r$ and $d(y,\lambda(t)) < r$. If $a,b \in [0,t+2r]$ such that $d(\lambda_{x,y}(a),\lambda_{x,y}(b)) < k$, then there is $a',b' \in [r,t+r]$ such that $|a-a'| \leqslant r$, $|b-b'| \leqslant r$, $d(\lambda_{x,y}(a),\lambda_{x,y}(a')\leqslant r$ and $d(\lambda_{x,y}(b),\lambda_{x,y}(b')) \leqslant r$. Then $d(\lambda_{x,y}(a'),\lambda_{x,y}(b')) \leqslant d(\lambda_{x,y}(a'),\lambda_{x,y}(a))+ d(\lambda_{x,y}(a),\lambda_{x,y}(b))+ d(\lambda_{x,y}(b),\lambda_{x,y}(b')) \leqslant k+2r$. So $|a'-b'| < s$, which implies that $|a-b| < s+2r$.

Thus $\Upsilon_{X}$ is a set of equicoarse arcs.
\end{proof}

The next proposition is an equivalence to the definition of perspective compactification that is useful in the rest of this section. It is a weak version of the Karlsson property.

\begin{prop}Let $(X,d)$ be a proper metric space and $X+_{f}W$ a compactification of $X$. Then $X+_{f}W\in Pers(\varepsilon_{d})$ if and only if for every $u \in \U_{f}$ and $t > 0$, there is a bounded set $S \subseteq X$ such that if $x,y \in X-S$, such that $d(x,y) < t$, then $(x,y)\in u$.
\end{prop}

\begin{proof}$(\Rightarrow)$ Let $t>0$ and $u \in \U_{f}$. Suppose that for every $S$ bounded subset of $X$, there are $x_{S},y_{S} \in X-S$ such that $d(x_{S},y_{S}) < t$ and $(x_{S},y_{S}) \notin u$. We have that the set $\{(x_{S},y_{S}): S$ is bounded$\}$ is an element of $\varepsilon_{d}$. Since the set $\{x_{S}: S$ is bounded$\}$ is not bounded, there is a net $\{x_{S_{\gamma}}\}_{\gamma \in \Gamma}$ that converges to a point $w \in W$. Since $X+_{f}W$ is perspective, the net $\{y_{S_{\gamma}}\}_{\gamma \in \Gamma}$ also converges to  $w$. So there is $\gamma_{0} \in \Gamma$ such that $\forall \gamma > \gamma_{0}$, $(x_{S_{\gamma}}, w) \in v$ and $(w, y_{S_{\gamma}}) \in v$, where $v$ is symmetric element of $\U_{f}$ such that $v^{2} \subseteq u$. Then $(x_{S_{\gamma}},y_{S_{\gamma}}) \in v^{2} \subseteq u$, a contradiction. Thus, for every $u \in \U_{f}$ and $t > 0$, there is a bounded set $S \subseteq X$ such that if $x,y \in X-S$, such that $d(x,y) < t$, then $(x,y)\in u$.

$(\Leftarrow)$ Let $e \in \varepsilon_{d}$. Since e is a proper metric space, $(X,\varepsilon_{d})$ is a proper coarse space, which implies that $e$ is proper. Let $\{(x_{\gamma},y_{\gamma})\}_{\gamma \in \Gamma} \subseteq e$ be a net such that $\{x_{\gamma}\}_{\gamma\in \Gamma}$ converges to $x \in W$. Since $e \in \varepsilon_{d}$, we have that $sup\{d(x_{\gamma},y_{\gamma}): \gamma \in \Gamma\} = t < \infty$. Let $u \in \U_{f}$. There exists a bounded set $S \subseteq X$ such that if $a,b \in X-S$ such that $d(a,b) < t+1$ , then $(a,b) \in u$. Since e is proper and $\{x_{\gamma}\}_{\gamma\in \Gamma}$ converges to a point in $W$, the sequence $\{y_{\gamma}\}_{\gamma\in \Gamma}$ has cluster points only in $W$ (Lemma \ref{clusterponitsontheboundary}). Then there exists $\gamma_{0}\in \Gamma$ such that $\forall \gamma > \gamma_{0}$, $x_{\gamma},y_{\gamma} \notin S$, and then $(x_{\gamma},y_{\gamma}) \in u$. So $\forall u \in \U_{f}$, there exists $\gamma_{0}\in \Gamma$ such that $\forall \gamma > \gamma_{0}$, $(x_{\gamma},y_{\gamma}) \in u$, which implies that $\{y_{\gamma}\}_{\gamma\in \Gamma}$  converges to $x$. Then $e$ is perspective, which implies that $X+_{f}W$ is perspective.
\end{proof}

\begin{lema}If $X+_{f}W$ is a perspective compactification of $X$ that satisfies the coarse Karlsson property with respect to $\Upsilon$, then it satisfies the coarse Karlsson property with respect to $\Upsilon_{X}$.
\end{lema}

\begin{proof}Let $u,v\in \U_{f}$ such that $v$ is symmetric and $v^{3} \subseteq u$. Since $X+_{f}W$ has the Karlsson property with respect to $\Upsilon$, there is a bounded set $S \subseteq X$ such that every $\lambda \in \Upsilon$ that do not intersect $S$ is $v$-small. Since $X+_{f}W$ has the perspectivity property, there is a bounded set $S' \subseteq X$ such that $\forall x \in X-S$, $\mathfrak{B}(x,r)$ is $v$-small. Let $\lambda_{x,y}: [0,t+2r] \rightarrow X$ be a coarse arc in $\Upsilon_{X}$ that do not intersect $S\cup S'$. We have that $Im \ \lambda_{x,y} = Im\lambda \cup \{x,y\}$, $(x,\lambda(0))\in v$ and $(\lambda(t),y) \in v$ (since $d(x,\lambda(0)) < r$ and $d(\lambda(t),y) < r$), which implies that $Im \ \lambda_{x,y}$ is $v^{3}$-small and then it is $u$-small. Thus $X+_{f}W$ satisfies the coarse Karlsson property with respect to $\Upsilon_{X}$.
\end{proof}

\begin{prop}Let $(X,d)$ and $(Y,d')$ be two proper metric spaces and $\pi: (Y,\varepsilon_{d'}) \rightarrow (X,\varepsilon_{d})$ a continuous coarse equivalence with a continuous quasi-inverse $\varpi$. If $X+_{f}W$ is a compactification with the coarse Karlsson property, then $Y+_{f^{\ast}}W$ has the coarse Karlsson property.
\end{prop}

\begin{proof}Let $\Upsilon$ be a set of equicoarse arcs on $X$ such that $X+_{f}W$ has the coarse Karlsson property with respect to. Then $\varpi(\Upsilon) = \{\varpi \circ \lambda: \lambda \in \Upsilon\}$ is a set of equicoarse arcs on $Y$.  By the last proposition, it is sufficient to show that $Y+_{f^{\ast}}W$ has the coarse Karlsson property with respect to $\varpi(\Upsilon)$. Let $u\in \U_{f^\ast}$. Since $\pi$ and $\varpi$ are continuous and $X+_{f}W$ is perspective, we have that $\varpi+id: X+_{f}W \rightarrow Y+_{f^{\ast}}W$ is continuous, because the pullback of $Y+_{f^{\ast}}W$ is $\Lambda(Y+_{f^{\ast}}W)$ and $\Lambda = \Pi^{-1}$. Then $(\varpi+id)^{-1}(u)\in \U_{f}$. Since $X+_{f}W$ has the coarse Karlsson property with respect to $\Upsilon$, there exists a bounded set $S \subseteq X$ such that every coarse arc $\lambda \in \Upsilon$ that does not intersect $S$ is  $(\varpi+id)^{-1}(u)$-small. Let $\lambda \in \Upsilon$ such that $\varpi \circ \lambda$ does not intersect $\varpi(S)$. Then $\lambda$ does not intersect $S$, which implies that $\lambda$ is $(\varpi+id)^{-1}(u)$-small. This means that $(Im \ \lambda)^{2} \subseteq (\varpi+id)^{-1}(u)$, which implies that $(Im \ \lambda)^{2} \subseteq \varpi^{-1}(u\cap Y^{2})$. Then $(Im \ \varpi\circ \lambda)^{2} \subseteq \varpi\circ \varpi^{-1}(u\cap Y^{2}) \subseteq u \cap Y^{2} \subseteq u$. So $Y+_{f^{\ast}}W$ has the coarse Karlsson property with respect to $\varpi(\Upsilon)$.
\end{proof}

\begin{cor}If $\pi$ and $\varpi$ are continuous, then the functor $\Pi$ preserves the coarse Karlsson property. \eod
\end{cor}

\begin{prop}\label{quotientkarlsson}Let $X$ be a proper metric space, $X+_{f}W$, $X+_{f'}Z$ compactifications of $X$, $\Upsilon$ a set of equicoarse arcs on $X$ and a continuous map $id+\pi: X+_{f}W \rightarrow X+_{f'}Z$. If $X+_{f}W$ has the Karlsson property with respect to $\Upsilon$, then $X+_{f'}Z$ has the Karlsson property with respect to $\Upsilon$.
\end{prop}

\begin{proof}Let $u' \in \U_{f'}$. Then $u = (id+\pi)^{-1}(u') \in \U_{f}$. Since   $X+_{f}W$ has the Karlsson property with respect to $\Upsilon$, there exists a bounded set $S \subseteq X$ such that every coarse arc in $\Upsilon$ that do not intersect $S$ is $u$-small. Let $\lambda \in \Upsilon$ that does not intersect $S$. Let $x,y \in Im \ \lambda$. Then $(x,y) \in u$, which implies that $(x,y) = ((id+\pi)(x),(id+\pi)(y)) \in u'$, which implies that $\lambda$ is $u'$-small. Thus, $X+_{f'}Z$ has the Karlsson property with respect to $\Upsilon$.
\end{proof}

\begin{obs}In particular, if $X+_{f}W$ has the coarse Karlsson property, then $X+_{f'}Z$ has the coarse Karlsson property.
\end{obs}

\subsection{Quotients}

\begin{prop}\label{quotientpersp}Let $(X,\varepsilon)$ be a locally compact paracompact Hausdorff space with a proper coarsely connected coarse structure and $X+_{f}Y \in Pers(\varepsilon)$. If $X+_{g}Z$ is a compactification of $X$ and there exists a continuous map $id+\pi: X+_{f}Y \rightarrow X+_{g}Z$, then $X+_{g}Z \in Pers(\varepsilon)$.
\end{prop}

\begin{obs}In other words, if a compactification of $X$ is perspective, then every Hausdorff quotient of it that preserves $X$ is also perspective.
\end{obs}

\begin{proof}Since $id+\pi$ is continuous, we have that $id: (X,\varepsilon_{f}) \rightarrow (X,\varepsilon_{g})$ is coarse (Proposition 2.33 of \cite{Ro}), which implies that $\varepsilon_{f} \subseteq \varepsilon_{g}$. So $\varepsilon \subseteq \varepsilon_{g}$, which implies that $X+_{g}Z \in Pers(\varepsilon)$.
\end{proof}

\begin{cor}\label{discretepersp}Let $(X,\varepsilon)$ be a discrete space with the discrete coarse structure (i.e. $\varepsilon = \{e \subseteq X\times X: \#e \cap (X\times X - \Delta X) < \aleph_{0}\}$). Then $Pers(\varepsilon) = T_{2}Comp(X)$.
\end{cor}

\begin{proof}We have that $\beta X \in Pers(\varepsilon)$ (Example 2.43 of \cite{Ro}). By the last proposition we have that every other compactification of $X$ is perspective, since it is a quotient of $\beta X$.
\end{proof}

\subsection{Freudenthal compactifications are perspective}

\begin{defi}\label{perspectivelyconnected}Let $(X,d)$ be a metric space. We say that $X$ is perspectively connected if $\forall \epsilon > 0$, there exists $\delta > 0$ such that if $a,b \in X$ such that $d(a,b) < \epsilon$, there exists a connected set $C \subseteq X$ such that $a,b \in C$ and $diam \ C < \delta$.
\end{defi}

\begin{obs}It is easy to see that perspectively connected spaces are connected and geodesic spaces are perspectively connected.
\end{obs}

\begin{prop}\label{freudenthalperspective}Let $(X,d)$ be a proper metric space that is locally connected and perspectively connected. Then the Freudenthal compactification $X+_{f}Ends(X)$ is perspective with respect to $\varepsilon_{d}$.
\end{prop}

\begin{proof}Let $e \in \varepsilon_{d}$. We have that e is proper since $(X,\varepsilon_{d})$ is proper. Let $\{(x_{\gamma},y_{\gamma})\}_{\gamma \in \Gamma} \subseteq e$ be a net such that $\{x_{\gamma}\}_{\gamma \in \Gamma}$ converges to $x \in Ends(X)$. Since $\{(x_{\gamma},y_{\gamma})\}_{\gamma \in \Gamma} \subseteq e$, there exists $\epsilon > 0$ such that, $\forall \gamma \in \Gamma$, $d(x_{\gamma},y_{\gamma}) < \epsilon$. Since $X$ is perspectively connected, there exists $\delta > 0$ such that $\forall a,b \in X$ such that $d(a,b) < \epsilon$, there exists a compact set $C_{a,b} \subseteq X$ such that $a,b \in C_{a,b}$ and $diam \ C_{a,b} < \delta$. Let $K$ be a compact subset of $X$. There exists $\gamma_{0} \in \Gamma$ such that $\forall \gamma > \gamma_{0}$, $d(x_{\gamma}, K) > \delta$. So for every $\gamma > \gamma_{0}$, $diam \ C_{x_{\gamma},y_{\gamma}} <\delta$, which implies that  $C_{x_{\gamma},y_{\gamma}} \cap K = \emptyset$ and then $x_{\gamma}$ and $y_{\gamma}$ are in the same connected component of $X-K$ (since  $C_{x_{\gamma},y_{\gamma}}$ is connected). Let $\psi_{K}: Ends(X) \rightarrow \pi_{0}^{u}(X-K)$ be the projection map. Since  $\{x_{\gamma}\}_{\gamma \in \Gamma}$ converges to $x$ in $X+_{f}Ends(X)$, we have that  $\{x_{\gamma}\}_{\gamma \in \Gamma}$ converges to $\psi_{K}(x)$ in $X+_{f_{K}}\pi_{0}^{u}(X-K)$. Then, there exists $\gamma_{1} >\gamma_{0}$ such that $\forall \gamma > \gamma_{1}$, $x_{\gamma}$ is contained in $\psi_{K}(x)$ (seen as a connected component of $X-K$). Since $y_{\gamma}$ is in the same connected component of $x_{\gamma}$, it follows that $\forall \gamma > \gamma_{1}$, $y_{\gamma} \in \psi_{K}(x)$, which implies that  $\{y_{\gamma}\}_{\gamma \in \Gamma}$ converges to $\psi_{K}(x)$ (seen as an element of $\pi_{0}^{u}(X-K)$) in $X+_{f_{K}}\pi_{0}^{u}(X-K)$. Since it works for every choice of compact $K \subseteq X$, we have that $\{y_{\gamma}\}_{\gamma \in \Gamma}$ converges to $x$ in $X+_{f}Ends(X)$. Then $e$ is perspective.

Thus $X+_{f}Ends(X)$ is perspective.
\end{proof}

Let's see that the hypothesis of perspective connectedness is necessary:

\begin{ex}\label{quasiisobutdiffends} Let $X$ be a subset of $\R^{2}$, with the induced metric, given by $X = \{(0,y): 0 \leqslant y \leqslant 1\}\cup \{(x,y): x \geqslant 0, y = 0, 1\}$. This space $X$ is homeomorphic to $\R$, which implies that $\#Ends(X) = 2$. Consider the sequence $\{(a_{n},b_{n})\}_{n \in \N}$, where $a_{n} = (n,1)$ and $b_{n} = (n,2)$. Since $d(a_{n},b_{n}) = 1$, we have that $\{(a_{n},b_{n})\}_{n \in \N} \in \varepsilon_{d}$. But $\{a_{n}\}_{n \in \N}$ and $\{b_{n}\}_{n \in \N}$ converge to the two different points of $Ends(X)$, which doesn't happen for perspective compactifications.

Observe also that perspective connectedness is not an invariant of coarse equivalence. Let $Y = X \cup \{(n,y): n \in \N, 0 \leqslant y \leqslant 1\}$. We have that $X$ and $Y$ are quasi-isometric, but $Y$ is perspectively connected and $X$ is not. Observe also that $Ends(Y)$ is not homeomorphic to $Ends(X)$, since $Y$ is $1$-ended.
\end{ex}

\begin{cor}Let $(X,d)$ be a proper metric space that is perspectively connected and locally connected. If $W$ is a totally disconnected space, then any compactification of $X$ of the form $X+_{f}W$ is perspective with respect to $\varepsilon_{d}$.
\end{cor}

\begin{proof}It follows from the proposition above and \textbf{Proposition \ref{quotientpersp}}.
\end{proof}

\begin{prop}Let $(X,d)$ and $(Y,d')$ be proper perspectively connected locally connected spaces that are coarse equivalent. Let the isomorphism of categories $\Pi: Pers(\varepsilon_{d}) \rightarrow Pers(\varepsilon_{d'})$ induced by the coarse equivalence between $X$ and $Y$. Then $\Pi$ sends the Freudenthal compactification of $X$ to the Freudenthal compactification of $Y$.
\end{prop}

\begin{obs}By \textbf{Proposition \ref{freudenthalperspective}}, both Freudenthal compactifications are perspective.
\end{obs}

\begin{proof}Let $X+_{f}Ends(X)$ be the Freudenthal compactification of $X$ and $Y+_{g}Z$ a compactification of $Y$ with $Z$ totally disconnected. By the corollary above, the compactification $Y+_{g}Z$ is perspective with respect to $\varepsilon_{d'}$. Then the space $\Pi^{-1}(Y+_{g}Z)$ is a compactification of $X$ with boundary $Z$, which implies that there is a unique map $id_{X}+\phi: X+_{f}Ends(X) \rightarrow  \Pi^{-1}(Y+_{g}Z)$ which is continuous (by \textbf{Proposition \ref{universalpropertyends}}). Then the map $id_{Y}+\phi = \Pi(id_{X}+\phi): \\ \Pi(X+_{f}Ends(X)) \rightarrow  \Pi(\Pi^{-1}(Y+_{g}Z)) = Y+_{g}Z$ is continuous. The uniqueness of $id_{Y}+\phi$ comes from the fact that $\Pi$ is an isomorphism. So $\Pi(X+_{f}Ends(X))$ is the Freudenthal compactification of $Y$.
\end{proof}

\begin{cor}\label{coarseequivimpliesendshomeo}Let $(X,d)$ and $(Y,d')$ be proper perspectively connected locally connected spaces. If $X$ and $Y$ are coarse equivalent, then $Ends(X)$ and $Ends(Y)$ are homeomorphic.
\end{cor}

\begin{obs}This result is well known, at least for proper geodesic spaces that are quasi-isometric (Proposition 8.29 of \cite{BH}).
\end{obs}

\begin{proof}By the last proposition, the isomorphism $\Pi: Pers(\varepsilon_{d}) \rightarrow Pers(\varepsilon_{d'})$ sends the Freudenthal compactification of $X$ to the Freudenthal compactification of $Y$. But $\Pi$ preserves boundaries. So $Ends(X)$ is homeomorphic to $Ends(Y)$.
\end{proof}

\begin{obs}Observe that the spaces in \textbf{Example \ref{quasiisobutdiffends}} are proper but are not geodesic.
\end{obs}

\subsection{Coproducts}
\label{coproducts}

Let $\{C_{i}\}_{i\in \Gamma}$ be a family of compact Hausdorff spaces, $X = \dot{\bigcup}_{i\in \Gamma}C_{i}$ with the coproduct topology and $\pi: X \rightarrow \Gamma$ such that $\forall i \in \Gamma$, $\pi(C_{i}) = i$. We have that $\pi$ is a topologically proper map, where $\Gamma$ has the discrete topology.

Let $\zeta$ be the discrete coarse structure of $\Gamma$. We have that $Pers(\zeta) = T_{2}Comp(\Gamma)$.  Let $\varepsilon = \{e \subseteq X \times X: \exists f \in \zeta: e \subseteq \pi^{-1}(f)\}$.

\begin{prop}$(X,\varepsilon)$ is a proper coarsely connected coarse space.
\end{prop}

\begin{proof}We have that $\Delta X \subseteq \pi^{-1}(\Delta \Gamma)$, which implies that $\Delta X \in \varepsilon$.

Let $e \in \varepsilon$. There exists $f \in \zeta$ such that $e \subseteq \pi^{-1}(f)$, which implies that $e^{-1}\subseteq \pi^{-1}(f^{-1})$ and then $e^{-1} \in \varepsilon$.

Let $e,e' \in \varepsilon$.  There exists $f,f' \in \zeta$ such that $e \subseteq \pi^{-1}(f)$ and $e' \subseteq \pi^{-1}(f')$. We have that $f \cup f' \in \zeta$ and $e\cup e' \subseteq \pi^{-1}(f\cup f')$, which implies that $e\cup e' \in \varepsilon$.

Let $e,e' \in \varepsilon$.  There exists $f,f' \in \zeta$ such that $e \subseteq \pi^{-1}(f)$ and $e' \subseteq \pi^{-1}(f')$. Let $(a,b) \in e'\circ e$. There exists $c \in X:$ $(a,c)\in e$ and $(c,b) \in e'$, which implies that $(\pi(a),\pi(c))\in f$ and $(\pi(c),\pi(b)) \in f'$ and then $(\pi(a),\pi(b)) \in f' \circ f$. So $(a,b) \in \pi^{-1}(f'\circ f)$, which implies that $e'\circ e \subseteq \pi^{-1}(f'\circ f)$ and then $e'\circ e \in \varepsilon$.

Let $e \in \varepsilon$ and $e' \subseteq e$. There exists $f \in \zeta$ such that $e \subseteq \pi^{-1}(f)$, which implies that $e' \subseteq \pi^{-1}(f)$. Then $e' \in \varepsilon$.

Thus $\varepsilon$ is a coarse structure.

Let $(a,b) \in X \times X$. We have that $\{(\pi(a),\pi(b))\} \in \zeta$, since $\Gamma$ is coarsely connected. So $(a,b) \in \pi^{-1}(\{(\pi(a),\pi(b))\})$ and then $\{(a,b)\} \in \varepsilon$. Thus $(X,\varepsilon)$ is coarsely connected.

Let $U\in \zeta$ be a neighborhood of $\Delta \Gamma$. We have that $\pi^{-1}(U) \in \varepsilon$ and, since $\pi$ is continuous, it is a neighborhood of $\Delta X$.

Let $A \subseteq X$ be a bounded set. So $A \times A \in \varepsilon$, which implies that there exists $f \in \zeta$ such that $A\times A \subseteq \pi^{-1}(f)$ and then $\pi(A)\times \pi(A) \subseteq f$. So $\pi(A)\times \pi(A) \in \zeta$, which implies that $\pi(A)$ is bounded and then topologically bounded. Since $\pi$ is proper, we have that $A$ must be topologically bounded.

Thus, $(X,\varepsilon)$ is a proper coarse space.
\end{proof}

We say that $\varepsilon$ is the coarse structure induced by the map $\pi: X \rightarrow \Gamma$.

\begin{prop}$\pi: (X,\varepsilon) \rightarrow (\Gamma, \zeta)$ is coarse.
\end{prop}

\begin{proof}We already have that $\pi$ is proper. Let $e \in \varepsilon$. There exists $f \in \zeta$ such that $e \subseteq \pi^{-1}(f)$, which implies that $\pi(e) \subseteq f$ and then $\pi(e) \in \zeta$. Thus, $\pi$ is coarse.
\end{proof}

Let $\iota: \Gamma \rightarrow X$ be a section of $\pi$.

\begin{prop}$\iota: (\Gamma, \zeta) \rightarrow (X,\varepsilon)$ is coarse.
\end{prop}

\begin{proof}Let $A \subseteq X$ be a compact set. Then, there exists $i_{1},...,i_{n} \in \Gamma$ such that $A \subseteq C_{i_{1}}\cup...\cup C_{i_{n}}$, which implies that $\iota^{-1}(A)$ is finite. Then $\iota$ is proper.

Let $f \in \zeta$. We have that $\iota(f) \subseteq \pi^{-1}(f) \in \varepsilon$, which implies that $\iota(f) \in \varepsilon$.

Thus $\iota$ is coarse.
\end{proof}

\begin{prop}$\pi$ and $\iota$ are quasi-inverses.
\end{prop}

\begin{proof}By the definition of $\iota$, we have that $\pi \circ \iota = id_{\Gamma}$. Since $\forall i\in \Gamma$, $\forall a \in C_{i}$, $\iota\circ \pi(a) \in C_{i}$, we have that $\{(a,\iota \circ \pi(a)): a \in X\} \subseteq \bigcup_{i\in \Gamma}C_{i}\times C_{i}  = \pi^{-1}(\Delta \Gamma)$, which implies that $\{(a,\iota \circ \pi(a)): a \in X\} \in \varepsilon$. So $\iota \circ \pi$ and $id_{X}$ are close.  Thus, $\pi$ and $\iota$ are quasi-inverses.
\end{proof}

\begin{teo}\label{teoremaprincipal}$Pers(\varepsilon)\cong T_{2}Comp(\Gamma)$. Furthermore, such isomorphism preserves the boundaries of the spaces. \eod
\end{teo}

\begin{proof}It follows from \textbf{Corollary \ref{discretepersp}} and \textbf{Theorem \ref{Teorema1}}.
\end{proof}

We have also a criterion to say when a compactification of $X$ is perspective with respect to $\varepsilon$:

\begin{prop}\label{persphauss}Let $X+_{f}Y\in T_{2}Comp(X)$ and $\sim = \Delta(X+_{f}Y) \cup \bigcup_{i\in \Gamma} \pi^{-1}(i)$. Then $X+_{f}Y\in Pers(\varepsilon)$ if and only if $X+_{f}Y/\!\sim$ is Hausdorff.
\end{prop}

\begin{proof}$(\Rightarrow)$ Let $\Pi: Pers(\zeta) \rightarrow Pers(\varepsilon)$ be the isomorphism induced by $\pi$. We have that $X+_{f}Y = \Pi(\Gamma+_{g}Y)$ for some $\Gamma+_{g}Y \in Pers(\zeta)$. So $f = g ^{\ast}$, which implies that the map $\pi+id: X+_{f}Y \rightarrow \Gamma+_{g}Y$ is continuous. There exists a bijective map $t: X+_{f}Y/\!\sim \rightarrow  \Gamma+_{g}Y$ such that the following diagram commutes:

$$ \xymatrix{ X+_{f}Y \ar[r]^{\pi+id} \ar[d]_{\rho} & \Gamma+_{g}Y   \\
           X+_{f}Y/\!\sim \ar[ru]^{t} & } $$

Where $\rho$ is the quotient map. By the universal property of the quotient map $\rho$, the map $t$ is continuous and, since $X+_{f}Y/\!\sim$ is compact and $\Gamma+_{g}Y$ is Hausdorff, it is a homeomorphism. Thus $X+_{f}Y/\!\sim$ is Hausdorff.

$(\Leftarrow)$ Let $e \in \varepsilon$ and $\{(x_{i},y_{i})\}_{i\in \Upsilon} \subseteq e$ a net such that $\{x_{i}\}_{i\in \Upsilon}$ converges to an element $x \in Y$. We have that $X+_{f}Y/\sim$ is Hausdorff, which implies that it is homeomorphic to $\Gamma+_{g}Y$ for some choice of $g$ such that the diagram commutes:

$$ \xymatrix{ X+_{f}Y \ar[r]^{\pi+id} \ar[d]_{\rho} & \Gamma+_{g}Y   \\
           X+_{f}Y/\!\sim \ar[ru]^{t} & } $$

Where $\rho$ is the quotient map and $t$ is a homeomorphism. We have that $\Gamma+_{g}Y$ is perspective, which implies that $\zeta \subseteq \zeta_{g}$. So $\{(\pi(x_{i}),\pi(y_{i}))\}_{i\in \Upsilon} \subseteq \pi(e) \in \zeta$ and $\{\pi(x_{i})\}_{i\in \Upsilon}$ converges to $x$ (because $\pi+id$ is continuous). Since $\zeta \subseteq \zeta_{g}$, we have that $\{\pi(y_{i})\}_{i\in \Upsilon}$ converges to $x$. Since $\#(\pi+id)^{-1}(x) = 1$, we have that $\{y_{i}\}_{i\in \Upsilon}$ also converges to $x$ (\textbf{Proposition \ref{liftnet}}).

Let $A$ be a bounded subset of $X$. We have that $\pi(A)$ is bounded and $\pi(e)$ is proper, which implies that $\B(\pi(A),\pi(e))$ is bounded. But $\pi(\B(A,e)) \subseteq \B(\pi(A),\pi(e))$, which implies that $\pi(\B(A,e))$ is bounded. Then the set $\pi^{-1}(\pi(\B(A,e)))$ is bounded, since the map $\pi$ is proper. Since $\B(A,e) \subseteq \pi^{-1}(\pi(\B(A,e)))$, we have that $\B(A,e)$ is bounded. Analogously $\B(A,e^{-1})$ is also bounded. Then $e$ is proper.

Thus $e \in \varepsilon_{f}$ and then  $X+_{f}Y\in Pers(\varepsilon)$.
\end{proof}

\begin{cor}Let $X$ be a Hausdorff locally compact space. If $X$ can be decomposed into a coproduct of any infinite cardinality of compact spaces, then $X$ has an infinite amount of non-equivalent compactifications. \eod
\end{cor}

As consequence, we have an easy proof that spaces like the first uncountable ordinal $[0,\omega_{1})$ and the deleted Tychonoff plank $[0,\omega_{1}]\times [0,\omega_{0}] - \{(\omega_{1},\omega_{0})\}$ cannot be decomposed into a coproduct of any cardinality of compact spaces, since their Stone-Cech compactifications coincide with the one-point compactifications (Example 19.13 of \cite{SW}).

Let $C$ be a compact Hausdorff space and let's consider that $\Gamma = \N$ and $\forall i\in \N$, there exists homeomorphisms $\eta_{i}: C \rightarrow C_{i}$ and $\eta'_{i}: C \rightarrow C'_{i}$, where $X = \dot{\bigcup}_{i\in \N}C_{i} =  \dot{\bigcup}_{i\in \N}C'_{i}$. Let's consider $\pi,\pi': X \rightarrow \N$ such that $\forall i \in \N$, $\pi(C_{i}) = i$ and $\pi'(C'_{i}) = i$ and $\varepsilon$ and $\varepsilon'$ the coarse structures on $X$ induced by $\pi$ and $\pi'$, respectively.

\begin{prop}(Pelczy$\acute{n}$ski, p. 87 of \cite{Pe}) Let $Z$ be a compact metrizable space. Then there exists, up to homeomorphisms, a unique compactification of $\N$ with boundary $Z$.
\end{prop}

\begin{prop}\label{homeoconvex}Let $Z$ be a compact metrizable space. If $X+_{f}Z\in Pers(\varepsilon)$ and $X+_{g}Z \in Pers(\varepsilon')$, then they are homeomorphic.
\end{prop}

\begin{obs}Such homeomorphism doesn't need to be a morphism of $Comp(X)$.
\end{obs}

\begin{proof}There exists an homeomorphism $\theta+id: \N+_{f_{\ast}}Z \rightarrow \N+_{g_{\ast}}Z$ (the pushforwards are taken with respect to the pairs $\pi$ and $id$ and $\pi'$ and $id$, respectively and they are Hausdorff since $X+_{f}Z$ and $X+_{g}Z$ are perspective and by \textbf{Lemma \ref{pushforwradpersp}}). Let $\eta: X \rightarrow X$ be the homeomorphism such that $\forall i \in \N$, $\forall x \in C_{i}$, $\eta(x) = \eta'_{\theta(i)}\circ \eta_{i}^{-1}(x)$. The following diagrams are commutative:

$$ \xymatrix{ X \ar[r]^{\eta} \ar[d]^{\pi} & X \ar[d]_{\pi'}  & & Z \ar[r]^{id} \ar[d]^{id} & Z \ar[d]_{id} \\
           \N \ar[r]^{\theta} & \N & & Z \ar[r]^{id} & Z} $$

So by the Cube Lemma the map $\eta+id: X+_{(f_{\ast})^{\ast}}Z \rightarrow X+_{(g_{\ast})^{\ast}}Z$ is a homeomorphism. Since those compactifications have the perspectivity property, $(f_{\ast})^{\ast} = f$ and $(g_{\ast})^{\ast} = g$ (again by \textbf{Lemma \ref{pushforwradpersp}}), which implies that the map $\eta+id: X+_{f}Z \rightarrow X+_{g}Z$ is a homeomorphism.
\end{proof}

\subsection{The Cantor set minus one point}

\begin{defi}(Topological quasiconvexity of an equivalence relation) Let $X$ be a Hausdorff compact space and $\sim$ an equivalence relation on $X$. Then $\sim$ is topologically quasiconvex if $\forall q \in X$, the equivalence class $[q]$ is closed and $\forall u \in \U$,  $\#\{[x]\subseteq X:$ $[x] \notin Small(u)\}< \aleph_{0}$, with $\U$ the only uniform structure compatible with the topology of $X$.
\end{defi}

This terminology comes after the concept of dynamic quasiconvexity in geometric group theory. The two definitions are related but it is not our objective to work with groups on this paper (this topic will be covered in \cite{So3}).

\begin{prop}\label{quaseconvexidadetop}Let $X$ be a Hausdorff compact space and $\sim$ a quasiconvex equivalence class on $X$. If $A \subseteq X/ \sim$, we define $\sim_{A} = \Delta^{2}X \cup\bigcup\limits_{[x] \in A} [x]^{2}$. Then $\forall A \subseteq X/ \sim$,  $X/\sim_{A}$ is Hausdorff. \end{prop}

\begin{proof}Let $(x,y) \in Cl_{X^{2}}(\sim_{A}) - \Delta^{2}X$. Since $X$ is Hausdorff, $\exists u \in \U:$ $(x,y) \notin u$ (with $\U$ the unique uniform structure compatible with the topology of $X$). Let $v \in \U$ such that $v$ is symmetric and $v^{5} \subseteq u$. Take $a = \B(x,v)\times \B(y,v)$. If $[q]\in Small(v)$ and $a \cap [q]^{2}\neq \emptyset$, then $\B(x,v) \cup [q],\B(y,v) \cup [q] \in Small(v^{3})$ which implies that $\B(x,v) \cup \B(y,v) \in Small(v^{5})$, absurd since $(x,y) \notin v^{5} \subseteq u$. So $a \cap (\Delta^{2}X \cup \bigcup \limits_{[q]\in Small(v)} [q]^{2}) = \emptyset$. From the topological quasiconvexity, we have that the set $F = \{[q]\in A: [q] \notin Small(v)\}$ is finite. Since $(x,y) \in  Cl_{X^{2}}(\sim_{A})$, for every $U \subseteq a$, neighbourhood of $(x,y)$,  $U \cap \sim_{A} \neq \emptyset$, which implies that $U \cap \bigcup \limits_{[q]\in F} [q]^{2} \neq \emptyset$. So $(x,y) \in Cl_{X^{2}}(\bigcup \limits_{[q]\in F} [q]^{2}) = \bigcup \limits_{[q]\in F} Cl_{X^{2}}([q]^{2}) = \bigcup \limits_{[q]\in F} [q]^{2}$, which implies that $(x,y) \in \sim_{A}$. So $\sim_{A}$ is closed. Since $X$ is Hausdorff compact and $\sim_{A}$ is closed, it follows from Aleksandrov Theorem that $X/\!\sim_{A}$ is Hausdorff.
\end{proof}

Let $K_{0}$ denote the Cantor set minus one point.

\begin{lema}Let $X$ be a Hausdorff locally compact Lindelöf 0-dimensional space and $\U$ a uniform structure compatible with the topology of $X$. Then $\forall u \in \U$, $\exists V$ a partition of $X$ by compact open $u$-small sets.
\end{lema}

\begin{proof}
We have that $\forall u \in \U$, $\exists U$ a clopen cover of $K_{0}$ by $u$-small sets. Since $X$ is 0-dimensional and locally compact, we are able to build a refinement $U'$ such that  every element is compact and open. Since $X$ is Lindelöf we are able to take a subcover of $U'$, $U'' = \{U_{i}\}_{i \in \N}$. We take $V_{1} = U_{1}$, $V_{i} = U_{i} - (U_{1} \cup ... \cup U_{i-1})$ for $i \in \N$ and $V = \{V_{i}\}_{i \in \N}$. We have that $V$ is a refinement of $U''$ by compact open sets and it is also a partition of $X$. Since $V$ is a refinement of $U$, it consists of $u$-small sets. So $\forall u \in \U$, $\exists V$ a countable partition of $X$ by compact open $u$-small sets.
\end{proof}

\begin{prop}Let $K_{0}+_{f}Z$ be a compactification of $K_{0}$ with $Z$ metrizable. So there exists $\{L_{i}\}_{i\in \N}$ a partition of $K_{0}$ by compact open subsets such that $K_{0}+_{f}Z$ is perspective with respect to the coarse structure induced by the quotient map $\pi: K_{0} \rightarrow \N$ defined by the relation $\bigcup_{i\in\N} L_{i}^{2}$.
\end{prop}

\begin{proof}Let $\U$ be the unique uniform structure compatible with the topology of $K_{0}+_{f}Z$. By the \textbf{Propositions \ref{persphauss} and \ref{quaseconvexidadetop}} it is enough to proof that there is a partition $K_{0}$ by compact open sets $\{L_{i}\}_{i\in \N}$ such that $\forall u  \in \U$, $\#\{j \in \N: L_{j} \notin Small(u) \} < \aleph_{0}$.

Since $K_{0}+_{f}Z$ is metrizable, there exists a base  $\{u_{i}\}_{i\in \N}$  of $\U$. Take $U = \{V_{i}\}_{i \in \N}$ a partition of $K_{0}$ by compact open $u_{1}$-small sets. We have that $\forall i \in \N$,  $V_{i}$ is a compact $0$-dimensional Hausdorff space. So it has a partition $U_{i} =\{ V_{i,1},...,V_{i,k_{i}}\}$ by compact open $u_{i}$-small sets. We have that $\forall i,j \in \N, \ V_{i,j}$ is open in $K_{0}+_{f}Z$, since it is open in $V_{i}$. It follows that $U' = \bigcup_{i \in \N}U_{i}$ is a partition of $K_{0}$ by compact open sets. Let $u \in \U$ and $L \in \{V \in U': \ V \notin Small(u)\}$. Let $i \in \N$ such that $u_{i} \subseteq u$. We have that $L \notin Small(u_{i})$, which implies that $L \in \bigcup\limits_{j = 1}^{i-1}U_{j}$ which is a finite set. Thus $\{V \in U': \ V \notin Small(u)\}$ is finite.
\end{proof}

\begin{teo}\label{Cantor}Let $Z$ be a compact metrizable space. Then there exists, up to homeomorphisms, a unique compactification of $K_{0}$ with boundary $Z$.
\end{teo}

\begin{proof}Let $K_{0}+_{f}Z$ and $K_{0}+_{g}Z$ be compactifications of $K_{0}$. Then, there exists $\{L_{i}\}_{i\in \N}$ and $\{L'_{i}\}_{i\in \N}$ partitions of $K_{0}$ such that the compactifications are perspective with respect to the coarse structure induced by the quotient maps induced by the partitions, respectively. So, by the \textbf{Proposition \ref{homeoconvex}}, there exists an homeomorphism between $K_{0}+_{f}Z$ and $K_{0}+_{g}Z$.
\end{proof}


\begin{thebibliography}{99}

\bibitem{Bo}
F. Borceux,
\textit{Handbook of Categorical Algebra 1 - Basic Category Theory}.
Encyclopedia of Mathematics and its Applications, Cambridge University Press, Great Britain, 1994.
Zbl 1143.18001 MR 1291599

\bibitem{Bou}
N. Bourbaki,
\textit{Elements of Mathematics. General Topology. Part 1}.
Hermann, Paris; Addison-Wesley Publishing Co., Reading, Mass.-London-Don Mills, Ont.,  1966.
Zbl 0301.54001 MR 0205210

\bibitem{BH}
M. R. Bridson and A. Haefliger,
\textit{Metric spaces of non-positiive curvature}.
Springer,  1999.

\bibitem{DK}
C. Drutu and M. Kapovich,
\textit{Geometric group theory}.
Colloquium Publications, American Mathematical Society,  2018.

\bibitem{En}
R. Engelking,
\textit{General Topology}.
Sigma Series in Pure Mathematics, Heldermann Verlag Berlin, 1989.
Zbl 0684.54001 MR 1039321

\bibitem{Fr}
H. Freudenthal, $\ddot{U}$ber die Ender topologischer Räume und Gruppen.
\textit{Math Z}~\textbf{33}
(1931), 692-713.
Zbl 0002.05603 MR 1545233

\bibitem{Ge2}
V. Gerasimov, Floyd maps for relatively hyperbolic groups.
\textit{Geometric and Functional Analysis}~\textbf{22}
(2012), 1361-1399.
Zbl 1276.20050 MR 2989436

\bibitem{GP}
V. Gerasimov and L. Potyagailo, Quasi-isometric maps and Floyd boundaries of relatively hyperbolic groups.
\textit{J. Eur. Math. Soc.}~\textbf{15}
(2013), 2115-2137.

\bibitem{GM}
C. R. Guilbault and M. A. Moran, Proper homotopy types and $\mathcal{Z}$-boundaries of spaces admitting geometric group actions.
\textit{Expositiones Mathematicae}~\textbf{37}
(2018), 292-313.

\bibitem{Ho}
H. Hopf, Ender offener Räume und unendliche diskontinuierliche Gruppen.
\textit{Coment. Math Helv.}~\textbf{16},
(1943), 81-100.
Zbl 0060.40008 MR 0010267

\bibitem{Ka}
A. Karlsson, Free subgroups of groups with non-trivial Floyd boundary.
\textit{Comm. Algebra}~\textbf{31}
(2003), 5361-5376.

\bibitem{Ke}
A. S. Kechris,
\textit{Classical Descriptive Set Theory}. Springer-Verlag, 1995.
Zbl 0819.04002 MR 1321597

\bibitem{KP}
A. Kock and T. Plewe, Glueing analysis for complemented subtoposes.
\textit{Theory and Applications of Categories}~\textbf{2.9}
(1996), 100-112.

\bibitem{Ne}
S. B. Niefield, Cartesian inclusion: locales and toposes.
\textit{Communications in Algebra}~\textbf{9.16}
(1981), 1639-1671.

\bibitem{Pe}
A. Pelczy$\acute{n}$ski, A remark on space $2^{X}$ for zero - dimensional $X$.
\textit{Bull. Pol.
Acad. Sci.}~\textbf{13}
(1965), 85-89.

\bibitem{Ro}
J. Roe,
\textit{Lectures on Coarse Geometry}.
University Lecture Series, American Mathematical Society, 2003.

\bibitem{So}
L. H. R. de Souza, A generalization of convergence actions.\\
Preprint~2019. arXiv 1903.11746 [math.GR]

\bibitem{So3}
L. H. R. de Souza, Equivariant blowing up of bounded parabolic points.\\
Preprint~2021. arXiv 2008.05822v3 [math.GR]

\bibitem{SW}
S. Willard,
\textit{General Topology}.
Addison-Wesley Series in Mathematics, Addison-Wesley Publishing Company, 1968.
Zbl 0205.26601 MR 0264581

\end{thebibliography}
\end{document}